\documentclass[12pt, twoside]{article}
\usepackage[a4paper, left= 25mm, textheight=660pt]{geometry} 
\usepackage[new]{old-arrows}
\usepackage{mathtools}
\usepackage[utf8]{inputenc}
\usepackage{amsfonts}
\usepackage{amsmath}
\usepackage{amsthm}
\usepackage{amssymb}
\usepackage{hyperref}
\usepackage{mathrsfs}
\usepackage{tikz}
\usetikzlibrary{shapes.geometric}
\usepackage{tikz-cd}
\usepackage{float}
\usepackage[
backend=biber,
style=alphabetic,
maxnames=10, minalphanames=6, maxalphanames=10
]{biblatex} 
\usepackage[nottoc,notlot,notlof]{tocbibind}
\usepackage{titlesec}

\titleformat{\subsubsection}[runin]
{\normalfont\bfseries}{\thesubsubsection}{1em}{}

\newcommand{\ZZ}{\mathbb Z}

\newcommand{\RR}{\mathbb R}
\newcommand{\GG}{\mathbb G}
\newcommand{\PP}{\mathbb P}

\newcommand{\NN}{\mathbb N}

\newcommand{\XX}{\mathbb X}

\renewcommand{\AA}{\mathbb A}
\newcommand{\cO}{\mathcal O}
\newcommand{\cL}{\mathcal L}
\newcommand{\cM}{\mathcal M}

\newcommand{\cF}{\mathcal F}

\newcommand{\cK}{\mathcal K}
\newcommand{\cX}{\mathcal X}

\newcommand{\cI}{\mathcal I}

\newcommand{\Wedge}{\bigwedge}
\renewcommand{\Tilde}{\widetilde}
\renewcommand{\Bar}{\overline}

\DeclareMathOperator{\SL}{SL}

\DeclareMathOperator{\Proj}{Proj}
\DeclareMathOperator{\Spec}{Spec}
\DeclareMathOperator{\Hilb}{Hilb}

\DeclareMathOperator{\LW}{LW}

\DeclareMathOperator{\Aut}{Aut}

\DeclareMathOperator{\ind}{ind}

\DeclareMathOperator{\trop}{trop}
\DeclareMathOperator{\sm}{sm}

\DeclareMathOperator{\Stab}{Stab}

\DeclareMathOperator{\pr}{pr}

\DeclareMathOperator{\WS}{WS}
\DeclareMathOperator{\SWS}{SWS}
\DeclareMathOperator{\Hom}{Hom}
\DeclareMathOperator{\colim}{colim}

\DeclareMathOperator{\opp}{op}

\DeclareMathOperator{\Sets}{\underline{Sets}}
\DeclareMathOperator{\ksch}{\underline{k-Sch}}

\numberwithin{equation}{subsection}

\newtheorem{theorem}{Theorem}[subsection]
\newtheorem*{theorem*}{Theorem}
\newtheorem{lemma}[theorem]{Lemma}
\newtheorem{proposition}[theorem]{Proposition}
\newtheorem{corollary}[theorem]{Corollary}
\theoremstyle{definition}
\newtheorem{definition}[theorem]{Definition}
\theoremstyle{remark}
\newtheorem{remark}[theorem]{Remark}

\begin{document}

\title{\Large{\textbf{EXPANSIONS FOR HILBERT SCHEMES OF POINTS ON SEMISTABLE DEGENERATIONS}}}
\author{\large{CALLA TSCHANZ}}
\date{}
\maketitle

\textbf{Abstract.} The aim of this paper is to extend the expanded degeneration construction of Li and Wu to obtain good degenerations of Hilbert schemes of points on semistable families of surfaces, as well as to discuss alternative stability conditions and parallels to the GIT construction of Gulbrandsen, Halle and Hulek and logarithmic Hilbert scheme constructions of Maulik and Ranganathan. We construct a good degeneration of Hilbert schemes of points as a proper Deligne-Mumford stack and show that it provides a geometrically meaningful example of a construction arising from the work of Maulik and Ranganathan.

\setcounter{tocdepth}{1}

\tableofcontents

\section{Introduction}

The study of moduli spaces is a central topic in algebraic geometry; among moduli spaces, Hilbert schemes form an important class of examples. They have been widely studied in geometric representation theory, enumerative and combinatorial geometry and as the two main examples of hyperk\"ahler manifolds, namely Hilbert schemes of points on K3 surfaces and generalised Kummer varieties. A prominent direction in this area is to understand the local moduli space of such objects and, in particular, the ways in which a degeneration of smooth Hilbert schemes may be given a modular compactification. 

For example, we may consider the geometry of relative Hilbert schemes on a degeneration whose central fibre has normal crossing singularities. We may then ask how the singularities of such a Hilbert scheme may be resolved while preserving certain of its properties or how it may be expressed as a good moduli space. This then becomes a compactification problem with respect to the boundary given by the singular locus. Historically, an important method used in moduli and compactification problems has been Geometric Invariant Theory (GIT). More recently, the work of Maulik and Ranganthan \cite{MR} has explored how methods of tropical and logarithmic geometry can be used to address such questions for Hilbert schemes. This builds upon previous work of Li \cite{Li} and Li and Wu \cite{LW} on expanded degenerations for Quot schemes and work of Ranganathan \cite{Ranganathan} on logarithmic Gromov-Witten theory with expansions.

Briefly stated, the aim of this paper is to provide explicit examples of such compactifications and explore the connections between these methods.

\subsection{Basic setup}
Let $k$ be an algebraically closed field of characteristic zero. Let $X\to C$ be a projective family of surfaces over a curve $C\cong \AA^1$ such that the total space is smooth and the central fibre $X_0$ has simple normal crossing singularities. At a triple point of the singular fibre, $X$ is étale locally given by $\Spec k[x,y,z,t]/(xyz-t)$. In this étale local model, the general fibres are smooth and the central fibre $X_0$ is given by three planes intersecting transversely in $\AA^3$. Throughout this work, these will be denoted $Y_1,Y_2$ and $Y_3$, given in local coordinates by $x=0,\ y=0$ and $z=0$ respectively. Let $X^\circ \coloneqq X\setminus X_0$, which lies over $C^\circ\coloneqq C\setminus \{0\}$. Given such a family $X\to C$, we will explore how techniques of expanded degenerations may be used to construct good compactifications of the relative Hilbert scheme of $m$ points $\Hilb^m(X^\circ/C^\circ)$.

\medskip
The aim is to construct a compactification which is flat over $C$ and in which all limit subschemes can be chosen to satisfy some transversality condition in some modification of $X_0$. In general, transversality will mean that the subschemes should be normal to the codimension 1 strata of the central fibre. This forces any interesting behaviour of the subschemes to occur on the smooth irreducible components of the modifications of $X_0$. In the case which interests us here, namely Hilbert schemes of points, it will just mean that we would like our subschemes to have support in the smooth loci of the fibres. We will refer to this condition throughout this work as the condition that the subschemes are \emph{smoothly supported}. The problem therefore is to construct \emph{expansions} (birational modifications of the central fibre of $X$ in a 1-parameter family) in which all limits needed to compactify $\Hilb^m(X^\circ/C^\circ)$ can be chosen to be smoothly supported. This allows us to break down the problem of studying Hilbert schemes of points on $X_0$ into smaller parts, by studying the products of Hilbert schemes of points on the irreducible components of the modifications of $X$. Moreover, this approach will allow us to construct compactifications as stacks which have good properties as moduli spaces. In particular, for all the compactifications constructed in this way, the data of each family of length $m$ zero-dimensional subschemes over $C$ is completely determined by its degenerate fibre, i.e.\ by its limit in the compactification. As we will mention in the following section, the work of Li and Wu only covers the case where the singular locus of $X_0$ is smooth. We would like to highlight that understanding how these problems work in general for simple normal crossings is quite powerful, as we can always use semistable reduction to reduce to this case.

\medskip
As is mentioned in Section \ref{further results}, this type of construction can be applied to construct type III degenerations of Hilbert schemes of points on K3 surfaces. This will be described in future work.

\subsection{Previous work in this area}

Expanded degenerations were first introduced by Li \cite{Li} and then used by Li and Wu \cite{LW} to study Quot schemes on degenerations $X\to C$, in the case where the singular locus of $X_0$ is smooth. They construct a stack of expansions $\mathfrak{C}$ and a family $\mathfrak{X}$ over it, which is a stack of modifications of $X_0$ in an expanded family, subject to some equivalence relations; among others a relation induced by a natural torus action on the modified fibres. They then impose a stability condition which cuts out families of subschemes of $\mathfrak{X}$ which meet the boundary of the special fibre in a transverse way. For each choice of Hilbert polynomial the family thus obtained is a proper Deligne-Mumford stack.

\medskip
Following on from \cite{LW}, Gulbrandsen, Halle and Hulek \cite{GHH} present a GIT version of the above construction in the case of Hilbert schemes of points. They construct an explicit expanded degeneration, i.e.\ a modified family over a larger base, whose fibres correspond to blow-ups of components of $X_0$ in the family. They present a linearised line bundle on this space for the natural torus action and they are able to show that in this case the Hilbert-Mumford criterion simplifies down to a purely combinatorial criterion. Using this, they impose a GIT stability condition which recovers the transverse zero-dimensional subschemes of Li and Wu and prove that the corresponding stack quotient is isomorphic to that of Li and Wu. A motivation for this work was to construct type II degenerations of Hilbert schemes of points on K3 surfaces. Indeed, type II good degenerations of K3 surfaces present these types of singularities in the special fibre, which is a chain of surfaces intersecting along smooth curves.

\medskip
There is more recent work of Maulik and Ranganathan \cite{MR}, building upon earlier ideas of Ranganathan \cite{Ranganathan} and results of Tevelev \cite{Tevelev}, in which they use techniques of logarithmic and tropical geometry to construct appropriate expansions of $X\to C$. This allows them to define moduli stacks of transverse subschemes starting from the case where $X_0$ is any simple normal crossing variety. They show that the stacks thus constructed are proper and Deligne-Mumford. For more details on this, see Section \ref{MR_section}.

\subsection{Main results}\label{further results}
Let $X \to C$ be a semistable degeneration of surfaces. In the following sections, we propose explicit constructions of expanded degenerations and stacks of stable length $m$ zero-dimensional subschemes on these expanded families, which we show to have good properties.

\medskip
We start by constructing expanded degenerations as schemes and discuss various GIT stability conditions on the corresponding relative Hilbert schemes of points. Unlike the situation in \cite{GHH}, however, a single GIT condition is not sufficient to obtain the desired stable locus for this problem. This is explained in Section \ref{Motivation}. We then define a stack of expansions $\mathfrak{C}$ and a family $\mathfrak{X}$ over it, which contains all expansions of $X$ which can be obtained using a specific sequence of blow-ups. The type of expansion of $X$ which can appear in $\mathfrak{X}$ is therefore greatly restricted, which in this case offers significant advantages, as we will see. We then describe how \textit{Li-Wu stability} (abbreviated LW stability) can be extended to this setting and define an alternative notion of stability, called \textit{smoothly supported weak strict stability} (abbreviated SWS stability), derived from GIT stability conditions. We may then construct the stacks $\mathfrak{M}^m_{\LW}$ and $\mathfrak{M}^m_{\SWS}$ of LW and SWS stable length $m$ zero-dimensional subschemes on $\mathfrak{X}$. Our first main results are the following.

\begin{theorem}
    The stacks $\mathfrak{M}^m_{\LW}$ and $\mathfrak{M}^m_{\SWS}$ are Deligne-Mumford and proper.
\end{theorem}

\begin{theorem}
    There is an isomorphism of stacks
    \[
    \mathfrak{M}^m_{\LW}\cong \mathfrak{M}^m_{\SWS}.
    \]
\end{theorem}

This construction has the benefit of being very straightforward compared to the other possible constructions solving this problem, as we will discuss later. The restrictive choices made in the construction of $\mathfrak{X}$ mean that LW or SWS stability are already sufficient conditions to make the stacks of stable objects be proper. This is somewhat unexpected; indeed, in general we will need to take an additional stability condition, as can be seen in \cite{MR} (see Section \ref{MR_section} for the role of Donaldson-Thomas stability in this problem).

\subsubsection*{Allowing for different choices of expansions.}
In this paper, we discuss only a specific choice of model for the Hilbert scheme of points which we call the canonical moduli stack. In upcoming work, we will investigate how these methods can be extended to describe other choices of models. We will consider an approach which parallels work of Kennedy-Hunt on logarithmic Quot schemes \cite{K-H}, as well as recover certain geometrically meaningful choices of moduli stacks arising from the methods of Maulik and Ranganathan \cite{MR}. In particular, we will discuss how tube components and Donaldson-Thomas stability enter the picture in these more general cases (see Section \ref{MR_section} for definitions).

\subsubsection*{Application to hyperk\"ahler varieties.}
We only consider here the property that $X$ is a degeneration of surfaces with a specific type of singularity in its special fibre. A natural question is to study the more specific case where $X$ is a type III good degeneration of K3 surfaces and try to construct a family of Hilbert schemes of points on $X$ which will be minimal in the sense of the minimal model program, meaning a good or dlt minimal degeneration (see \cite{Nagai} and \cite{KLSV} for definitions of the minimality conditions). The singularities arising in such a degeneration $X$ are of the type described here, i.e.\ we can restrict ourselves to the local problem where $X_0$ is thought of as given by $xyz=0$ in $\AA^3$. Among other reasons, Hilbert schemes of points on K3 surfaces are interesting to study because they form a class of examples of hyperk\"ahler varieties.

\subsection{Organisation}

We start, in Section \ref{Background on tropical perspective}, by giving some background on logarithmic and tropical geometry, and an overview of the work of Maulik and Ranganathan from \cite{MR} which we will want to refer to in later sections.
Then, in Section \ref{Second_constr}, we set out an expanded construction on schemes and, in~\ref{GIT stability}, we discuss how various GIT stability conditions can be defined on this construction. In Section \ref{stack construction}, we describe a corresponding stack of expansions and family over it, building on the expanded degenerations we constructed as schemes. In Section \ref{canonical moduli stack}, we extend our stability conditions to this setting. We then show that the stacks of stable objects defined have the desired Deligne-Mumford and properness properties.

\subsubsection*{Acknowledgements.}
I would like to thank Gregory Sankaran for all his support throughout this project. Thank you also to my PhD examiners, Alastair Craw and Dhruv Ranganathan, for their many helpful comments. This work was undertaken while funded by the University of Bath Research Studentship Award. I am also grateful to Patrick Kennedy-Hunt and Thibault Poiret for many interesting conversations.

\section{Background on tropical perspective}\label{Background on tropical perspective}

We briefly introduce here the language of tropical and logarithmic geometry in the context of this problem. For more details on the contents of this section, see the article \cite{Abramovich}, lecture notes \cite{gael}, as well as the first section of \cite{MR}.

\subsection{Tropicalisation and expansion}\label{useful facts about trop geom}

\subsubsection*{Tropicalisation.}
Let $(X,\cM_X)$ be a logarithmic scheme, where the sheaf of monoids $\cM_X$ gives the divisorial logarithmic structure with respect to some simple normal crossing divisor $D\subset X$. Explicitly, for an open subset $U \subseteq X$, the sheaf  $\cM_X$ is given by
\[
\cM_{X}(U) \coloneqq \{ f\in \cO_X(U) \ | \ f|_{U\setminus D} \in \cO^*_X(U\setminus D) \}.
\]
Then we can associate a fan $\Sigma_X$ to this in the following way. Recall that the characteristic sheaf $\Bar{\cM}_X$ for the divisorial logarithmic structure is defined by
\[
\Bar{\cM}_X(U) = \{ f\in \cO_X(U) \ | \  f|_{U\setminus D} \in \cO^*_X(U\setminus D), \ f|_D = 0 \},
\]
i.e.\ it records the data of monomials which vanish at the divisor $D$. The purpose of $\Sigma_X$ will be to keep track of the corresponding degrees of vanishing. We let
\[
\Sigma_X \coloneqq \colim_{x\in X} (\Bar{\cM}_{X,x})^\vee.
\]
This will be contained in $\RR^r_{\geq 0}$, where $r$ is the number of components of $D$, since $D$ is a simple normal crossing divisor. We call $\Sigma_X$ the \textit{tropicalisation} of $X$.

\subsubsection*{Subdivisions of the tropicalisation define expansions of $X$.}
In the following, we will want to study possible birational modifications of the scheme $X$ around the divisor $D$. In the tropical language, these are expressed as subdivisions.

\begin{definition}
    Let $\Upsilon$ be a fan, let $|\Upsilon|$ be its support and $\upsilon$ be a continuous map
    \[
    \upsilon\colon |\Upsilon| \longrightarrow \Sigma_X
    \]
    such that the image of every cone in $\Upsilon$ is contained in a cone of $\Sigma_X$ and that is given by an integral linear map when restricted to each cone in $\Upsilon$. We say that $\upsilon$ is a \textit{subdivision} if it is injective on the support of $\Upsilon$ and the integral points of the image of each cone $\tau\in\Upsilon$ are exactly the intersection of the integral points of $\Sigma_X$ with $\tau$.
\end{definition}

A subdivision of the tropicalisation defines a birational modification of $X$ in the following way. The subdivision
\[
\Upsilon \longhookrightarrow \Sigma_X \longhookrightarrow \RR^r_{\geq 0}
\]
has an associated toric variety $\AA_{\Upsilon}$, which comes with a $\GG^r_m$-equivariant birational map $\AA_\Upsilon \to \AA^r$. Then we have an induced morphism of quotient stacks
\[
[\AA_\Upsilon/\GG_m^r] \longrightarrow [\AA^r/\GG_m^r]
\]
and we may define the induced birational modification of $X$ to be
\[
X_\Upsilon \coloneqq X \times_{[\AA^r/\GG_m^r]} [\AA_\Upsilon/\GG_m^r].
\]
We shall call such a birational modification an \textit{expansion} of $X$.

\subsubsection*{Visualising the problem.}
Here, we describe how to visualise the tropicalisation arising from the divisorial logarithmic structure on $X$ associated to a simple normal crossing divisor $D\subset X$. We explain this for the case which interests us here, that is, we assume that $X\to C$ is locally given by $\Spec k[x,y,z,t]/(xyz-t)$ and the boundary divisor is $D\coloneqq X_0$.

\medskip
Given a divisorial logarithmic structure on $X$, the tropicalisation is a fan or cone complex which for each defining function of the divisor records the degree of vanishing of this function in $X$. Here, the functions vanishing at $D$ will be $x,y$ and $z$, therefore we may represent $\Sigma_X$ as a fan in $\RR^3_{\geq 0}$, in this case the positive orthant and its faces, as can be seen in Figure \ref{trop(X)}. In this image the three half-lines correspond to the divisors $Y_1,\ Y_2$ and $Y_3$ in $X$. The 2-dimensional faces spanned by two such lines correspond to the intersections $Y_i\cap Y_j$ and the three dimensional interior of the cone corresponds to the triple intersection point $Y_1\cap Y_2\cap Y_3$. For convenience, we shall refer to this tropicalisation as $\trop(X)$ in later sections.

\begin{figure}
    \centering
    \begin{tikzpicture}[scale=3]
    
        \coordinate (O) at (0,0,0);
      \draw[thick,->] (0,0,0) -- (1,0,0) node[below= 5, right=0]{$Y_2$};
      \draw[thick,->] (0,0,0) -- (0,1,0) node[right=1]{$Y_1$};
      \draw[thick,->] (0,0,0) -- (0,0,1) node[above=2]{$Y_3$};
      \draw (0.5,0.5,0) node[anchor=center]{$Y_1\cap Y_2$};
      \draw (0.5,0,0.5) node[anchor=center]{$Y_2\cap Y_3$};
      \draw (0,0.7,0.9) node[anchor=center]{$Y_1\cap Y_3$};
      \filldraw[black] (0,0,0) circle (0.7pt) ;
    \end{tikzpicture}
        \caption{Tropicalisation of $X$.}
        \label{trop(X)}
\end{figure}
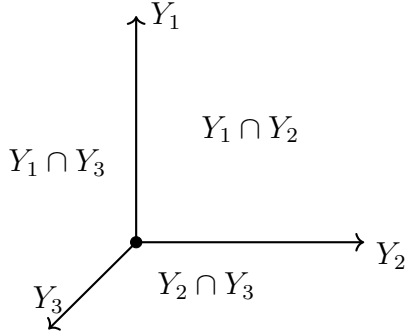

\medskip
The tropicalisation of $D$ can be visualised by taking a hyperplane slice through the cone in Figure \ref{trop(X)}; this yields a triangle with vertices corresponding to $Y_1,Y_2$ and $Y_3$ in $X_0$, edges between these vertices corresponding to the lines $Y_i\cap Y_j$, and 2-dimensional interior corresponding to the point $Y_1\cap Y_2\cap Y_3$, as pictured in Figure \ref{trop(X_0)}. This also corresponds to the \emph{dual complex} of the fibre $X_0$ (see \cite{dFKX} for a definition). We shall refer to the tropicalisation of $X_0$ as $\trop(X_0)$.

\medskip
Recall that $C\cong \AA^1$ and the fan of $\AA^1$ is a half-line with a distinguished vertex. Making a choice of point on this line corresponds to choosing a height for the triangle within the cone $\RR^3_{\geq 0}$. Geometrically, we can think of changing the height of the triangle as making a finite base change on $X$.

\begin{figure}
    \centering
    \begin{tikzpicture}[scale=3]
    
        \coordinate (O) at (0,0,0);
      \draw[thick,->] (0,0,0) -- (1,0,0) node[below= 5, right=0]{$y$};
      \draw[thick,->] (0,0,0) -- (0,1,0) node[right=1]{$x$};
      \draw[thick,->] (0,0,0) -- (0,0,1) node[above=2]{$z$};
      \draw[thick] (0,0,0.5) -- (0.5,0,0) node[below= 7, right=0]{$Y_2$};
      \draw[thick] (0.5,0,0) -- (0,0.5,0) node[right=1]{$Y_1$};
      \draw[thick] (0,0.5,0) -- (0,0,0.5) node[above=5, left = -1]{$Y_3$};
      \filldraw[black] (0.5,0,0) circle (0.7pt) ;
      \filldraw[black] (0,0.5,0) circle (0.7pt) ;
      \filldraw[black] (0,0,0.5) circle (0.7pt) ;
      \filldraw[black] (0,0,0) circle (0.5pt) ;
    \end{tikzpicture}
        \caption{Tropicalisation of $X_0$.}
        \label{trop(X_0)}
\end{figure}
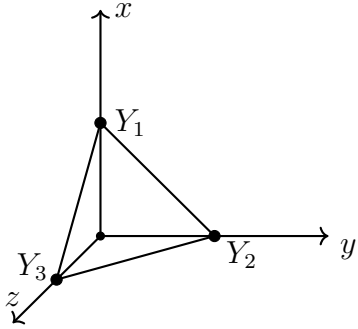

\subsection{Maulik-Ranganathan construction}\label{MR_section}

We will briefly recall some key points of \cite{MR}.
The aim of their work is to study the moduli space of ideal sheaves of fixed numerical type which meet the boundary divisor transversely. Some key motivations for the study of such an object come from enumerative geometry. For example, a common method used to address problems of curve counting in a given smooth variety is to degenerate this variety to a singular union of simpler irreducible components. The property of transversality is then crucial to ensure that all interesting behaviour of the ideal sheaves on the degenerate object occurs with support in the interior of the simpler irreducible components, which allows us to study it with more ease. One of the main difficulties with this approach is that often, as in this setting, the space of transverse ideal sheaves with respect to $D$ is non-compact. Constructing the appropriate compactification will yield a space which is flat and proper over $C$. In \cite{MR}, Maulik and Ranganathan formulate the Donaldson-Thomas theory of the pair $(X,D)$, starting by constructing compactifications of the space of ideal sheaves in $X$ transverse to $D$.

\medskip
We discuss \cite{MR} specifically with respect to the case which interests us here, namely that of a degeneration $X\to C$ as described above, where we seek to study the moduli space of ideal sheaves with fixed constant Hilbert polynomial $m$, for some $m\in \NN$ with respect to the boundary divisor $D= X_0$. The key idea is to construct the tropicalisation of $X$, denoted $\Sigma_X$, and a corresponding tropicalisation map, which is used to understand how to obtain the desired transversality properties in our compactifications.

\subsubsection*{Tropicalisation map.}
We may construct a tropicalisation map which takes points of $X^\circ$ to $\Sigma_X$, as in Section 1.4 of \cite{MR}. We recall the details of this map here. We assume that $\mathcal{K}$ is a valued field extending $k$. First, we take a point of $X^\circ(\cK)$, given by some morphism $\Spec \cK \to X^\circ$. By the properness of $X$, this extends to a morphism $\Spec R \to X$ for some valuation ring $R$. Now, let $P\in X$ denote the image of the closed point by the second morphism. The stalk of the characteristic sheaf at $P$ is given by $\NN^r$, where $r$ is the number of linearly independent vanishing equations of $D$ at the point $P$. For example, in our context, if $P\in Y_1\subset X_0$, then $r = 1$ and $\NN$ is generated by the function $x$; if $P\in Y_1\cap Y_2$, then $r=2$ with $\NN^2$ generated by the functions $x$ and $y$; etc.

Each element of $\NN^r$ corresponds to a function $f$ on $X$ in the neighbourhood of $P$ up to multiplication by a unit and we may then evaluate $f$ with respect to the valuation map associated to $\cK$. This determines an element of
\[
[\NN^r \to \RR_{\geq 0}] \in \Hom(\NN^r,\RR_{\geq 0}) \longhookrightarrow \Sigma_X.
\]
This gives rise to a morphism
\[
\trop \colon X^\circ \longrightarrow \Sigma_X
\]
called the tropicalisation map. Now let the valuation map $\cK\to \RR$ be surjective and let $Z^\circ\subset X^\circ$ be an open subscheme. We denote by $\trop(Z^\circ)$ the image of the map $\trop$ restricted to $Z^\circ(\cK)$. Maulik and Ranganathan are then able to show, based on previous work of Tevelev \cite{Tevelev} for the toric case, that given such an open subscheme $Z^\circ\subset X^\circ$, the subset $\trop(Z^\circ)$ gives rise to an expansion $X'$ of $X$ in which the closure $Z$ of $Z^\circ$ has the required transversality properties. This gives us a convenient dictionary to move back and forth between the geometric and combinatorial points of view.

\medskip
The possible tropicalisations of such subschemes, corresponding to expansions on the geometric side, are captured on the combinatorial side by the notion of \emph{1-complexes} embedding into $\Sigma_X$. See \cite{MR} for precise definitions.

\subsubsection*{Existence and uniqueness of transverse limits.}
Maulik and Ranganathan introduce notions of \textit{dimensional transversality} and \textit{strong transversality}, which, in the specific case of Hilbert schemes of points, happen to be equivalent to Li-Wu stability (see Section \ref{stab_conditions} for a definition of this stability condition). In general for higher dimensional subschemes this will not be the case, however.

\medskip
As mentioned above, for an open subscheme $Z^\circ\subset X^\circ$, we may consider its image $\trop(Z^\circ)$ under the tropicalisation map. Now recall from Section \ref{useful facts about trop geom} that a subdivision of the tropicalisation $\Sigma_X$ defines an expansion of $X$. The expansion corresponding to the subdivision given by $\trop(Z^\circ)$ in $\Sigma_X$ will in general determine a birational modification, not necessarily a blow-up. We note here that while contracting components of a fibre over a 1-parameter family will in general be flat, this is no longer the case when this operation is made over a larger base. It will therefore be necessary, for each possible $\trop(Z^\circ)$, to make a choice of polyhedral subdivision corresponding to an actual blow-up on $X$. Maulik and Ranganathan prove that, given any such $Z^\circ$, an expansion can be constructed from $\trop(Z^\circ)$ which has the required transversality property and good existence and uniqueness properties.

\subsubsection*{Construction of the stacks of expansions.}
Through these methods, one can obtain existence and uniqueness properties for these flat limits. To build a moduli space of such subschemes, Maulik and Ranganathan start by constructing a moduli space of possible expansions arising from Tevelev's procedure. Let us denote the set of isomorphism classes of 1-complexes which embed into $\Sigma_X$ by $|T(\Sigma_X)|$. Some subtleties arise at this point, namely that in general the space constructed will not be representable as a logarithmic algebraic stack. This can be seen through the fact the category of logarithmic algebraic stacks is equivalent to the category of cone stacks, but $|T(\Sigma_X)|$ cannot in general be given a proper cone structure. In order to give it a cone structure, Maulik and Ranganathan study the spaces of maps $\XX_G$ from the graphs $G$ associated to 1-complexes in $|T(\Sigma_X)|$ to $\Sigma_X$, and identify maps which have the same image. By taking appropriate subdivisions of the objects $\XX_G$ and identifying them in the right way, they obtain a \emph{moduli space of tropical expansions} $T$, which has the desired cone structure.

\medskip
This operation results in non-uniqueness, as we are making a choice of polyhedral subdivision and there is in general no canonical choice.

\subsubsection*{Proper Deligne-Mumford stacks.}

In order to construct the universal family $\mathfrak{Y} \subset T\times \Sigma$, some additional choices must be made. Indeed, as mentioned above, $\trop(Z^\circ)$ does not in general define a blow-up, so, when fitting the expansions we constructed together into one large family over a larger base, we must modify these expansions to ensure flatness. Here this is resolved by adding distinguished vertices to the relevant complexes. These added vertices will be 2-valent vertices along edges of the 1-complexes parameterised by $T$ and we call them \emph{tube vertices}. Geometrically, they look like $\PP^1$-bundles over curves in $X_0$ (where we took $X$ to be a family of surfaces). Again, this operation is not canonical and results in non-uniqueness.

\medskip
The addition of these tube vertices in the tropicalisation means that there are more potential components in each expansion, which interferes with the previously set up uniqueness results. Indeed, recall that $\trop(Z^\circ)$ gave us exactly the right number of vertices in the dual complex in order for each family of subschemes $Z^\circ \subset X^\circ$ to have a unique limit representative. Therefore, to reflect this, \emph{Donaldson-Thomas stability} asks for subschemes to be DT stable if and only if they are tube schemes precisely along the tube components. We say that a 1-dimensional subscheme is a \emph{tube} if it is the schematic preimage of a zero-dimensional subscheme in $D$. In the case of Hilbert schemes of points, this condition will translate simply to a 0-dimensional subscheme $Z$ being DT stable if and only if no tube component contains a point of the support of $Z$ and every other irreducible component expanded out by our blow-ups contains at least one point of the support of $Z$. 

\medskip
Maulik and Ranganathan define a subscheme to be \emph{stable} if it is strongly transverse and DT stable. For fixed numerical invariants the substack of stable subschemes in the space of expansions forms a Deligne-Mumford, proper, separated stack of finite type over~$C$.

\subsubsection*{Comparison with the results of this paper.}
The construction we present in this paper has the surprising property that we do not need to label any components as tubes in order for the stack of stable objects we define to be proper. This is an artifact of the specific choices of blow-ups to be included in our expanded degenerations. The work of Maulik and Ranganathan shows us that this is not expected in general. As mentioned in Section \ref{further results}, we will discuss in an upcoming paper how to construct proper stacks of stable objects in cases where different choices of expansions are made and it becomes necessary for us as well to introduce a Donaldson-Thomas stability condition.

\section{The expanded construction}\label{Second_constr}

In this section we construct explicit expanded degenerations $X[n]$ out of a 1-parameter family $X\to C$ by expanding the base and making sequences of blow-ups on the expanded family. As we will see these support a global action by the torus $G\coloneqq \GG_m^n$. We construct these spaces as schemes here. Later, in Section \ref{stack construction}, we give a stack construction building upon these schemes, in which we impose additional equivalence relations which essentially set to be equivalent any two fibres which look identical. We will touch more upon why this is necessary in Section \ref{stack construction}.

\subsubsection*{Setup and assumptions.}

As before, let $X\to C$ be a family of surfaces over a curve isomorphic to $\AA^1$, where $X$ is given in étale local coordinates by $\Spec k [x,y,z,t]/(xyz-t)$. We denote by $X_0$ the special fibre and by $Y_1$, $Y_2$ and $Y_3$ the irreducible components of this special fibre given locally by $x=0, y=0$ and $z=0$ respectively. Figure \ref{geom and trop special fibre} shows a copy of the special fibre $X_0$ both from the geometric point of view, on the left, and tropical point of view, on the right.

\begin{figure} 
    \begin{center}   
    \begin{tikzpicture}[scale=1.4]
        \draw   
        (0,0) -- (0,2)       
        (-1.732, -1) -- (0,0)
        
        (1.732, -1) -- (0,0);
        \draw (-1.3,1) node[anchor=center]{$Y_1$};
        \draw (1.3,1) node[anchor=center]{$Y_2$};
        \draw (0,-1) node[anchor=center]{$Y_3$};
        \draw (0,2.2) node[anchor=center]{$Y_1\cap Y_2$};
        \draw (-1.832, -1.2) node[anchor=center]{$Y_1\cap Y_3$};
        \draw (1.832, -1.2) node[anchor=center]{$Y_2\cap Y_3$};
        \draw (0.2, 0.2) node[anchor=center]{$Y_1\cap Y_2 \cap Y_3$};
    \end{tikzpicture}
    \hspace{2cm}
    \begin{tikzpicture}[scale=1.2]
\draw   

        (-1.732, -1) -- (0,2)       
        (-1.732, -1) -- (1.732, -1)
        
        (1.732, -1) -- (0,2);
        \draw (-1.4,0.8) node[anchor=center]{$Y_1\cap Y_2$};
        \draw (1.4,0.8) node[anchor=center]{$Y_2\cap Y_3$};
        \draw (0,-1.3) node[anchor=center]{$Y_1\cap Y_3$};
        \filldraw[black] (0,2) circle (2pt) ;
        \filldraw[black] (-1.732, -1) circle (2pt) ;
        \filldraw[black] (1.732, -1) circle (2pt) ;
        \draw (0,2.3) node[anchor=center]{$Y_1$};
        \draw (-1.832, -1.3) node[anchor=center]{$Y_3$};
        \draw (1.832, -1.3) node[anchor=center]{$Y_2$};
        \draw (0, 0) node[anchor=center]{$Y_1\cap Y_2 \cap Y_3$};
\end{tikzpicture}
    \end{center}
    \caption{Geometric and tropical pictures of the special fibre $X_0$.}
    \label{geom and trop special fibre}
\end{figure}
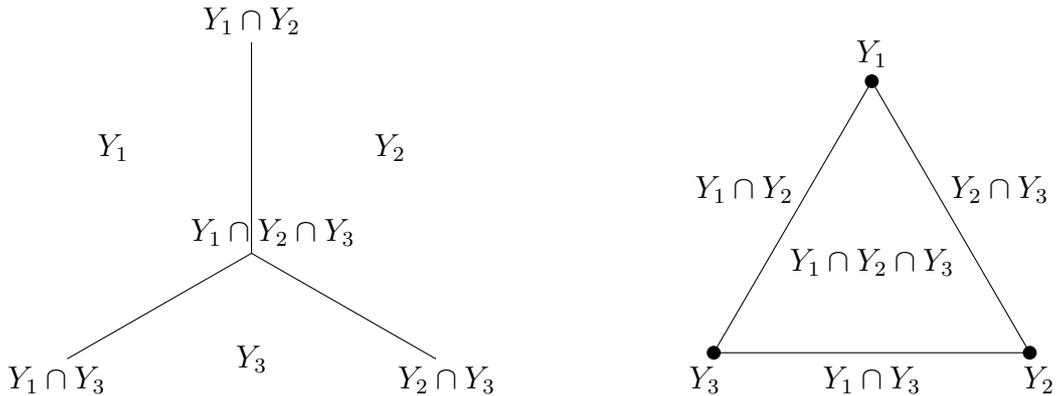

\subsubsection*{Output of expanded construction.} The expanded degeneration $X[n]\to C[n]$ which we construct in this section has the following properties:
\begin{itemize}
    \item The morphism $X[n]\to C[n]$ is projective and $G$-equivariant.
    \item Etale locally, $X[n]$ is a subvariety of $(X\times_{\AA^1} \AA^{n+1})\times (\PP^1)^{2n}$.
\end{itemize}

\subsection{The blow-ups}\label{BUs}

In the following, we construct expanded degenerations by enlarging the base $C$ and making sequences of blow-ups in the family over this larger base. We start by taking a copy of $\AA^{n+1}$, with elements labelled $(t_1, \ldots, t_{n+1}) \in \AA^{n+1}.$ Throughout this work, we shall refer to the entries $t_i$ as \textit{basis directions}. Now, let $X\times_{\AA^1} \AA^{n+1}$ be the fibre product given by the map $X\to C\cong \AA^1$ and the product
\[
(t_1, \ldots, t_{n+1}) \longmapsto t_1\cdots t_{n+1}.
\]
In this expanded degeneration construction, we will be blowing up schemes along Weil divisors. A consequence of the way these blow-ups are defined is that the blow-up morphisms contract only components of codimension at least 2.

\subsubsection*{First blow-up of the $Y_1$ component.}

 We start by blowing up $Y_{1} \times_{\AA^1} V(t_1)$ inside $X\times_{\AA^1} \AA^{n+1}$, where $V(t_i)$ denotes the locus where $t_i=0$.

 \medskip
 \emph{Notation.} We name the space resulting from this blow-up $X_{(1,0)}$ to signify we have blown up the component $Y_1$ once and the component $Y_2$ zero times.

 \medskip
 We can describe this blow-up locally in the following way. The ideal of the blow-up is $I_1=\langle x,t_1\rangle$. Globally this will correspond to an ideal sheaf $\cI_1$. Then there is a surjective map of graded rings
\[
A[x_0^{(1)},x_1^{(1)}] \longrightarrow S_1 = \bigoplus_{n\geq 0} I_1^n
\]
which maps
\[
x_0^{(1)}\longmapsto x \ \textrm{ and } \ x_1^{(1)}\longmapsto t_1,
\]
where $A\coloneqq k[x,y,z,t_1,\ldots, t_{n+1}]/(xyz - t_1\cdots t_{n+1})$. This induces an embedding 
\[
\Proj(S_1) \longhookrightarrow \Proj A[x_0^{(1)},x_1^{(1)}] = \PP^1\times\Spec A 
\]
and $\Proj(S_1)$, i.e.\ our blow-up, is cut out in $\PP^1\times\Spec A $ by the equations
\begin{align*}
    &x_0^{(1)}t_1 = xx_1^{(1)} \\
    &x_0^{(1)} yz = x_1^{(1)} t_2  \cdots t_{n+1} .
\end{align*}

\begin{proposition}\label{prop 1st BU}
    $X_{(1,0)}$ is isomorphic to $X\times_{\AA^1} \AA^{n+1}$ away from the locus where $t_1=t_i = 0$, for any $i\neq 1$.
\end{proposition}

\begin{proof}
    Let $X_{(1,0)}\to \AA^{n+1}$ be the natural projection. Then the fibres above $(t_1,\ldots, t_{n+1})$ where $t_1$ is nonzero are still the same after the blow-up and so are the fibres where $t_1=0$ and all the other $t_i$ are nonzero because the total space is still smooth at all points of these fibres. However, when $t_1=0$ and at least one of the other $t_i$ is zero, then we get singularities of the total space appearing in the fibre of $X\times_{\AA^1} \AA^{n+1}$ and the blow-up causes a new component to appear around the $Y_1$ component.
\end{proof}

 \emph{Notation.} We denote by $\Delta_1^{(1)}$ the new component introduced by the blow-up which is described in the proof of Proposition \ref{prop 1st BU} above.

 \medskip
 This can be seen in Figure \ref{geom and trop picture of (x,t1) BU}, where the added red vertices in the tropical picture correspond to the two irreducible components of $\Delta_1^{(1)}$ and the edge connecting them corresponds to the intersection of these irreducible components.

\begin{figure} 
    \begin{center}   
    \begin{tikzpicture}[scale=1.4]
        \draw   
        (0,0) -- (0,2)       
        (-1.732, -1) -- (0,0)
        (2.6, -1.5) -- (0,0)
        (0.866,-0.5) -- (0.866,2)
        (0.866,-0.5) -- (-1.732,-2)
        ;
        \draw (-1.3,1) node[anchor=center]{$Y_1$};
        \draw (1.8,1) node[anchor=center]{$Y_2$};
        \draw (0.5,-1.5) node[anchor=center]{$Y_3$};
        \draw (0.433,1) node[anchor=center, color = red]{$\Delta_1^{(1)}$};
        \draw (-0.866,-1) node[anchor=center, color = red]{$\Delta_1^{(1)}$};
        
    \end{tikzpicture}
    \hspace{2cm}
    \begin{tikzpicture}[scale=1.2]
\draw   

        (-1.732, -1) -- (0,2)       
        (-1.732, -1) -- (1.732, -1)
        
        (1.732, -1) -- (0,2)
        
        ;
        \draw[red] (-0.866,0.5) -- (0.866,0.5);
        \draw (-1.3,0.5) node[anchor=center, color = red]{$\Delta_1^{(1)}$};
        \draw (1.4,0.5) node[anchor=center, color = red]{$\Delta_1^{(1)}$};
        \filldraw[red] (-0.866,0.5) circle (2pt) ;
        \filldraw[red] (0.866,0.5) circle (2pt) ;
        \filldraw[black] (0,2) circle (2pt) ;
        \filldraw[black] (-1.732, -1) circle (2pt) ;
        \filldraw[black] (1.732, -1) circle (2pt) ;
        \draw (0,2.3) node[anchor=center]{$Y_1$};
        \draw (-1.832, -1.3) node[anchor=center]{$Y_3$};
        \draw (1.832, -1.3) node[anchor=center]{$Y_2$};
\end{tikzpicture}
    \end{center}
    \caption{Geometric and tropical pictures of a fibre in $X_{(1,0)}$ where $t_1=t_i=0$.}
    \label{geom and trop picture of (x,t1) BU}
\end{figure}
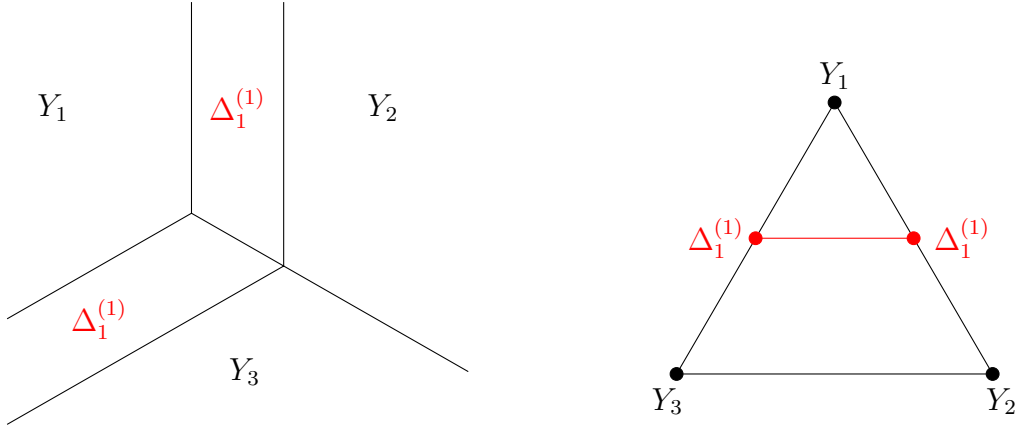

\subsubsection*{Further blow-ups of the $Y_1$ component.}
Let $b_{(1,0)} \colon X_{(1,0)} \to X\times_{\AA^1}\AA^{n+1} $ be the map defined by the first blow-up given above. We then proceed to blow-up $b_{(1,0)}^*(Y_{1} \times_{\AA^1} V(t_2))$ inside $X_{(1,0)}$. We name the resulting space $X_{(2,0)}$ and the composition of both blow-ups is denoted $b_{(2,0)} \colon X_{(2,0)} \to X\times_{\AA^1}\AA^{n+1}$. We continue to blow up each $b_{(k-1,0)}^*(Y_{1} \times_{\AA^1} V(t_k))$ inside $X_{(k-1,0)}$ for each $k\leq n$. The resulting space is denoted $X_{(n,0)}$. Finally, we denote by 
\[
\beta^1_{(k,0)} \colon X_{(k,0)} \longrightarrow X_{(k-1,0)}
\]
the morphisms corresponding to each individual blow-up. We therefore have the equality
\[
\beta^1_{(k,0)}\circ \cdots \circ \beta^1_{(1,0)} = b_{(k,0)}
\]
We now fix the following terminology.

\begin{definition}
    We say that a dimension 2 component in a fibre of $X_{(k,0)}\to C\times_{\AA^1} \AA^{n+1}$ is a $\Delta_1$-\textit{component} if it is contracted by the morphism $\beta^1_{(i,0)}$ for some $i\leq k$. Moreover if a $\Delta_1$-component in a fibre is contracted by such a map then we say it is \textit{expanded out} in this fibre.
\end{definition}

We label by $\Delta_1^{(k)}$ the $\Delta_1$-component resulting from the $k$-th blow-up. The fibre where $t_i=0$ for all $i\in \{1,\ldots, n+1\}$ has exactly $n$ expanded  $\Delta_1$-components. The equations of the blow-ups in local coordinates are as follows:

\begin{align}\label{BU1 eqns}
    &x_0^{(1)}t_1 = xx_1^{(1)}, \nonumber \\
    &x_1^{(k-1)}x_0^{(k)}t_k = x_0^{(k-1)}x_1^{(k)}, \qquad \textrm{ for } \ 2\leq k\leq n, \\
    &x_0^{(n)} yz = x_1^{(n)} t_{n+1}. \nonumber
\end{align}

\begin{remark}
    If we restrict $X_0$ to only the components $Y_1$ and $Y_2$, i.e.\ restrict the original degeneration to $\Spec k[x,y,z,t]/(xy-t)$, we get back exactly the blow-ups of Gulbrandsen, Halle and Hulek \cite{GHH}.
\end{remark}

\medskip
In fibres of the construction where $\Delta_1^{(k)}$, for some $k$, is not expanded out, i.e.\ not contracted by some map $\beta^1_{(i,0)}$, we will want to think of it in the following way.
\begin{definition}\label{equal to}
   When $t_k=0$ and all other $t_i$ are nonzero, we consider $\Delta_1^{(k)}$ and all $\Delta_1^{(j)}$ for $j\geq k$ as being \textit{equal to}  $Y_1$, meaning that the projective coordinates introduced by the $j$-th blow-up are proportional to $1/yz$. This follows from the equality
\[
x_0^{(j)} yz = x_1^{(j)} t_{j+1} \cdots t_{n+1},
\]
obtained from the above equations of the blow-ups. Similarly, the components $\Delta_1^{(j)}$ with $j<k$ are considered to be \textit{equal to} the union $Y_2\cup Y_3$, which follows from the equality
\[
x_0^{(j)}t_1\cdots t_j = xx_1^{(j)},
\]
obtained from the equations of the blow-up. When $t_{n+1}=0$ and all other $t_k$ are nonzero, then all $\Delta_1^{(k)}$ are \textit{equal to} the union $Y_2\cup Y_3$. 
\end{definition}

\subsubsection*{Blow-ups of the $Y_2$ component.}

For the component $Y_2$ we can make similar definitions to the above. We blow up $b_{(n,0)}^*Y_{2} \times_{\AA^1} V(t_{n+1})$ in $X_{(n,0)}$ and name the resulting space $X_{(n,1)}$. Let $b_{(n,k)} \colon X_{(n,k)} \to X\times_{\AA^1}\AA^{n+1} $ be the composition of the $n$ blow-ups of $Y_1$ and the first $k$ blow-ups of $Y_2$ on $X_{(n,0)}$. Similarly to the above, but with the order of the basis directions reversed, we blow up $b_{(n,k-1)}^*(Y_{2} \times_{\AA^1} V(t_{n+2-k}))$ in $X_{(n,k-1)}$ for each $k\leq n$.

\medskip
The equations of the blow-ups in local coordinates are as follows, where $(y_0^{(k)}:y_1^{(k)})$ are the coordinates of the $\PP^1$ introduced by the $k$-th blow-up:
\begin{align}\label{BU2 eqns}
    &y_0^{(1)}t_{n+1} = yy_1^{(1)}, \nonumber \\
    &y_1^{(k-1)}y_0^{(k)}t_{n+2-k} = y_0^{(k-1)}y_1^{(k)} \ \textrm{ for } \ 2\leq k\leq n, \\
    &y_0^{(n)} xz = y_1^{(n)} t_{1} \nonumber \\
    &x_0^{(k)} y_0^{(n+1-k)} z = x_1^{(k)} y_1^{(n+1-k)}. \nonumber
\end{align}

\emph{Notation.} The components introduced by these new blow-ups are labelled $\Delta_2^{(k)}$. To simplify notation, we will denote the base $\AA^{n+1}\times_{\AA^1} C$ by $C[n]$, the expanded construction $X_{(n,n)}$ by $X[n]$ and the natural projection to the original family $X$ by $\pi: X[n] \to X$.

\medskip
We have blow-up morphisms
\begin{align*}
    &\beta^1_{(i,j)} \colon X_{(i,j)} \longrightarrow X_{(i-1,j)}, \\
    &\beta^2_{(i,j)} \colon X_{(i,j)} \longrightarrow X_{(i,j-1)},
\end{align*}
corresponding to each individual blow-up of a pullback of the $Y_1$-component and $Y_2$-component respectively. The composition of all the blow-up morphisms is denoted
\[
b\coloneqq \beta^2_{(n,n)}\circ \cdots \circ\beta^2_{(n,1)}\circ \cdots \circ \beta^1_{(n,0)}\circ \cdots \circ \beta^1_{(1,0)} \colon X[n]\to X\times_{\AA^1} \AA^{n+1}.
\]
As the following proposition shows, the spaces $X_{(i,j)}$ are well-defined, as the order in which we make the blow-ups, i.e.\ expand out the $\Delta_1$ or the $\Delta_2$-components first, makes no difference. We can therefore express the space $X_{(m_1,m_2)}$ as the space $X_{(m_1,0)}$ on which we perform a sequence of blow-ups of the pullback of $Y_2$ or as the space $X_{(0,m_2)}$ on which we perform a sequence of blow-ups of the pullback of $Y_1$, etc.

\begin{proposition}
    The following blow-up diagram commutes
    \begin{figure}[H]
    \centering
        \begin{tikzcd}
        &\vdots \arrow[d] &\vdots \arrow[d] &\vdots \arrow[d] \\
        \cdots \arrow[r] &X_{(2,2)} \arrow[r,"\beta^2_{(2,2)}"] &X_{(2,1)} \arrow[r,"\beta^2_{(2,1)}"] \arrow[d,"\beta^1_{(2,1)}"] &X_{(2,0)} \arrow[d,"\beta^1_{(2,0)}"] \\
        \cdots \arrow[r] &X_{(1,2)} \arrow[r,"\beta^2_{(1,2)}"] &X_{(1,1)} \arrow[r,"\beta^2_{(1,1)}"] \arrow[d,"\beta^1_{(1,1)}"] &X_{(1,0)} \arrow[d, "\beta^1_{(1,0)}"] \\
         \cdots \arrow[r] &X_{(0,2)} \arrow[r,"\beta^2_{(0,2)}"] &X_{(0,1)} \arrow[r,"\beta^2_{(0,1)}"] &X_{(0,0)} = X\times_{\AA^1}\AA^{n+1}.
        \end{tikzcd}
        \label{fig:my_label}
    \end{figure}
    \end{proposition}

\begin{proof}
    We show that the space $X[1] = X_{(1,1)}$ can be constructed by first blowing up along $Y_1$ and then $Y_2$ or by reversing the order of these operations. Indeed, if we start by blowing up $Y_{1} \times_{\AA^1} V(t_{1})$ in $X\times_{\AA^1}\AA^{n+1}$, we obtain the étale local equations \eqref{BU1 eqns}. This gives us the space $X_{(1,0)}$. Then blowing up $b_{(1,0)}^*Y_{2} \times_{\AA^1} V(t_{2})$ in $X_{(1,0)}$ yields the étale local equations \eqref{BU2 eqns} and by definition this gives us the space $X_{(1,1)}$.

    \medskip
    Now, if we start by blowing up $Y_{2} \times_{\AA^1} V(t_{2})$ in $X\times_{\AA^1}\AA^{n+1}$, we obtain étale local equations
    \begin{align*}
        &y_0^{(1)}t_{n+1} = yy_1^{(1)},  \\
        &y_0^{(1)} xz = y_1^{(1)} t_{1}
    \end{align*}
    and this yields the space $X_{(0,1)}$. If we then blow up $b_{(0,1)}^*Y_{1} \times_{\AA^1} V(t_{1})$ in $X_{(0,1)}$,
    we shall obtain the equations
    \begin{align*}
        &x_0^{(1)} y_0^{(1)} z = x_1^{(1)} y_1^{(1)} \\
        &x_0^{(1)}t_1 = xx_1^{(1)}, \\
        &x_0^{(1)} yz = x_1^{(1)} t_{2}.
    \end{align*}
    But these are exactly the equations \eqref{BU1 eqns} and \eqref{BU2 eqns}, so the resulting space is again $X[1] = X_{(1,1)}$. This argument can be easily generalised to $X[n]$ for any $n$.
\end{proof}

\begin{proposition}
    If we take $X\to C$ to be the étale local model
    \[
    \Spec k[x,y,z,t]/(xyz-t) \longrightarrow \Spec k[t],
    \]
    the corresponding scheme $X[n]$ obtained after the sequence of blow-ups $b$ is a subvariety of $(X\times_{\AA^1} \AA^{n+1})\times (\PP^1)^{2n}$ cut out by the local equations \eqref{BU1 eqns} and \eqref{BU2 eqns}.
\end{proposition}

\begin{proof}
    This is immediate from the local description of the blow-ups above.
\end{proof}

\begin{proposition}
    The family $X[n]\to C[n]$ thus constructed is projective.
\end{proposition}

\begin{proof}
    The morphism $X\times_{\AA^1}\AA^{n+1} \to C[n]$ must be projective since $X\to C$ is projective. Then $X[n]\to X\times_{\AA^1}\AA^{n+1}$ is just a sequence of blow-ups along Weil divisors, hence projective. This proves projectivity of the morphism $X[n]\to C[n]$.
\end{proof}

\begin{remark}
    The issue with projectivity in Proposition 1.10 of \cite{GHH} only arises if the local descriptions of the blow-ups they use to create the family $X[n]\to C[n]$ do not glue globally to define blow-ups.
\end{remark}

We now extend the definition of $\Delta_1$-components to the schemes $X[n]$ and fix some additional terminology.
\begin{definition}\label{expanded definition}
    We say that a dimension 2 component of $X[n]\to C[n]$ is a $\Delta_i$-\textit{component} if it is contracted by the morphism $\beta^i_{(j,k)}$ for some $i, j, k$. Moreover if a $\Delta_i$-component in a fibre is contracted by such a map then we say it is \emph{expanded out in this fibre}. We say that a dimension 2 component of $X[n]$ is a $\Delta$-\textit{component} if it is a $\Delta_i$-component for some $i$. If it is expanded out in some fibre we may alternatively refer to it as an \textit{expanded component}. Similarly, we may extend Definition \ref{equal to} to say that a $\Delta$-component is \textit{equal to} a component $W$ of a fibre of $X[n]$ if the projective coordinates associated to this $\Delta$-component are proportional to the non-vanishing coordinates of $W$.
\end{definition}

\begin{definition}
    We say that a $\Delta_i$-component is of \textit{pure type} if it is not equal to any $\Delta_j$-component for $j\neq i$. Otherwise we say it is of \textit{mixed type}.
\end{definition}

\subsubsection*{Description of fibres of $X[n]\to C[n]$.}
In order to understand what these blow-ups look like, we describe the fibres of the scheme $X[n]$ over $C[n]$, where certain basis directions vanish.

\medskip
\emph{Only one basis direction vanishes.} If only one of the $t_i=0$ and all other basis directions are nonzero, then a fibre over such a point in the base is just a copy of the special fibre $X_0$.

\medskip
\emph{Two basis directions vanish.} Here, we consider fibres where $t_i=t_j=0$ for some $i<j$ and no other $t_k = 0$. The blow-ups of pullbacks of the $Y_1$-component cause exactly one $\Delta_1$-component to be expanded in such a fibre, and this expanded component is given by $\Delta_1^{(i)} = \ldots = \Delta_1^{(j-1)}$. In this case, the singularities of the total space occurring at the intersection of $Y_1$ and $Y_2$ have already been resolved by expanding out this $\Delta_1$-component. As the blow-ups of pullbacks of the $Y_2$-component also cause one $\Delta_2$-component to be expanded in this fibre, given by $\Delta_2^{(n+2-j)} = \ldots = \Delta_2^{(n+1-i)}$, we therefore have
\[
\Delta_1^{(i)} = \ldots = \Delta_1^{(j-1)} = \Delta_2^{(n+2-j)} = \ldots = \Delta_2^{(n+1-i)}
\]
in the $\pi^*((Y_1\cap Y_2)^\circ)$ locus of the fibre. This can be easily deduced from studying the equations of the blow-ups. In the $\pi^*((Y_1\cap Y_3)^\circ)$ locus of the fibre, we see a single expanded component of pure type given by $\Delta_1^{(i)} = \ldots = \Delta_1^{(j-1)}$. Similarly, in the $\pi^*((Y_2\cap Y_3)^\circ)$ locus of the fibre, we see a single expanded component of pure type given by $\Delta_2^{(n+1-j)} = \ldots = \Delta_2^{(n+1-i)}$. Finally, the component $\Delta_1^{(k)}$ is equal to the union $Y_2\cup Y_3$ for $k<i$ and $\Delta_1^{(l)}$ is equal to the component $Y_1$ if $l>j-1$. The situation for the $\Delta_2$ components is similar. This can be seen in Figure \ref{12}.

\begin{figure} 
    \begin{center}   
    \begin{tikzpicture}[scale=1.4]
        \draw   (-0.866,0.5) -- (0,0)
        (0,0) -- (0.866,0.5)       
        (-0.866,0.5) -- (-0.866,2)
        (0.866,0.5) -- (0.866,2)
        (-1.732, -1) -- (0,0)
        (-2.165,-0.249) -- (-0.866,0.5)
        
        (1.732, -1) -- (0,0)
        
        (2.165,-0.249) -- (0.866,0.5);
        \draw (-1.516,-0.375) node[anchor=center]{$\Delta_1^{(i)}$};
        \draw (1.516,-0.375) node[anchor=center]{$\Delta_2^{(j)}$};
        \draw (0,1.25) node[anchor=center]{$\Delta_1^{(i)} = \Delta_2^{(j)}$};
        \draw (-1.7,1) node[anchor=center]{$Y_1$};
        \draw (1.7,1) node[anchor=center]{$Y_2$};
        \draw (0,-0.8) node[anchor=center]{$Y_3$};
        
    \end{tikzpicture}
    \hspace{2cm}
    \begin{tikzpicture}[scale=1.4]
\draw   

        (-1.732, -1) -- (0,2)       
        (-1.732, -1) -- (1.732, -1)
        
        (1.732, -1) -- (0,2)
        
        ;
        \draw[red] (-0.866,0.5) -- (0.866,0.5);
        \draw[red] (0,-1) -- (0.866,0.5);
        \draw (-1.3,0.5) node[anchor=center, color = red]{$\Delta_1^{(i)}$};
        \draw (1.8,0.5) node[anchor=center, color = red]{$\Delta_1^{(i)} = \Delta_2^{(j)}$};
        \draw (0,-1.3) node[anchor=center, color = red]{$\Delta_2^{(j)}$};
        \filldraw[red] (-0.866,0.5) circle (2pt) ;
        \filldraw[red] (0.866,0.5) circle (2pt) ;
        \filldraw[black] (0,2) circle (2pt) ;
        \filldraw[black] (-1.732, -1) circle (2pt) ;
        \filldraw[black] (1.732, -1) circle (2pt) ;
        \filldraw[red] (0, -1) circle (2pt) ;
        \draw (0,2.3) node[anchor=center]{$Y_1$};
        \draw (-1.832, -1.3) node[anchor=center]{$Y_3$};
        \draw (1.832, -1.3) node[anchor=center]{$Y_2$};
\end{tikzpicture}
    \end{center}
    \caption{Geometric and tropical picture at $t_i=t_j=0$ in $X[n]$.}
    \label{12}
\end{figure}
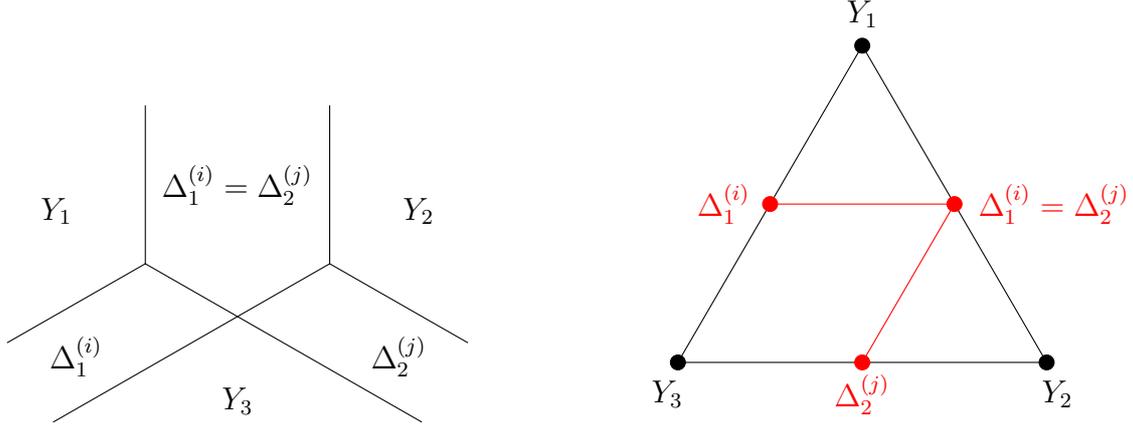

Before we continue we fix some terminology which will help us describe the expanded components.

\begin{definition}
    We refer to an irreducible component of a $\Delta$-component as a \textit{bubble}. The notions of two bubbles being \textit{equal} and a bubble being \textit{expanded out} in a certain fibre are as in Definitions \ref{equal to} and \ref{expanded definition}.
\end{definition}

\begin{figure} 
    \begin{center}   
    \begin{tikzpicture}[scale=1.6]
        \draw    (0.866,-0.5) -- (0,-1)
        (0,-1) -- (-0.866,-0.5)
        (0,0) -- (0,1.5)
        (0,0) -- (-0.866,-0.5)
        (0,0) -- (0.866,-0.5)
        (-0.866,-0.5) -- (-1.732,0)       
        (-1.732,0) -- (-1.732,1.5)
        (-1.732, -1) -- (-0.866,-0.5)
        (1.732, -1) -- (0.866,-0.5)
        (-0.866,-1.5) -- (0,-1)
        (0.866,-1.5) -- (0,-1)
        (-2.599, -0.5) -- (-1.732,0)
        (1.732, 0) -- (1.732,1.5)
        (1.732, 0) -- (0.866,-0.5)
        (1.732, 0) -- (2.599,-0.5)
        ;
        \draw (-1.732,-0.375) node[anchor=center]{$\Delta_1^{(j)}$};
        \draw (-0.866,-0.875) node[anchor=center]{$\Delta_1^{(i)}$};
        \draw (1.732,-0.375) node[anchor=center]{$\Delta_2^{(j)}$};
        \draw (0.866,-0.875) node[anchor=center]{$\Delta_2^{(k)}$};
        \draw (0.9,1) node[anchor=center]{$\Delta_1^{(i)} = \Delta_2^{(j)}$};
        \draw (-0.9,1) node[anchor=center]{$\Delta_1^{(j)} = \Delta_2^{(k)}$};
        \draw (-2.2,1) node[anchor=center]{$Y_1$};
        \draw (2.2,1) node[anchor=center]{$Y_2$};
        \draw (0,-1.5) node[anchor=center]{$Y_3$};
        \draw (0,-0.5) node[anchor=center]{$ \Delta_1^{(i)}= \Delta_2^{(k)}$};
        
    \end{tikzpicture} \\
    \vspace{1cm}
    \hspace{1.5cm}
    \begin{tikzpicture}[scale=1.3]
\draw   

        (-1.732, -1) -- (0,2)       
        (-1.732, -1) -- (1.732, -1)
        
        (1.732, -1) -- (0,2)
        
        ;
        \draw[red] (-0.433,1.25) -- (0.433,1.25);
        \draw[red] (-0.866,-1) -- (0.433,1.25);
        \draw[red] (-1.3,-0.25) -- (1.3,-0.25);
        \draw[red] (0.866,-1) -- (1.3,-0.25);
        \draw (-1.7,-0.25) node[anchor=center, color = red]{$\Delta_1^{(i)}$};
        \draw (-0.833,1.25) node[anchor=center, color = red]{$\Delta_1^{(j)}$};
        \draw (1.4,1.25) node[anchor=center, color = red]{$\Delta_1^{(j)}= \Delta_2^{(k)}$} ;
        \draw (2.3,-0.25) node[anchor=center, color = red]{$\Delta_1^{(i)} = \Delta_2^{(j)}$};
        \draw (-0.18,0) node[anchor=center, color = red]{$\Delta_1^{(i)} = \Delta_2^{(k)}$};
        \draw (0.866,-1.3) node[anchor=center, color = red]{$\Delta_2^{(j)}$};
        \draw (-0.866,-1.3) node[anchor=center, color = red]{$\Delta_2^{(k)}$};
        \filldraw[red] (-1.3,-0.25) circle (2pt) ;
        \filldraw[red] (1.3,-0.25) circle (2pt) ;
        \filldraw[red] (0.866,-1) circle (2pt) ;
        \filldraw[red] (-0.866,-1) circle (2pt) ;
        \filldraw[black] (0,2) circle (2pt) ;
        \filldraw[black] (-1.732, -1) circle (2pt) ;
        \filldraw[black] (1.732, -1) circle (2pt) ;
        \filldraw[red] (-0.433,1.25) circle (2pt) ;
        \filldraw[red] (0.433,1.25) circle (2pt) ;
        \filldraw[red] (-0.433,-0.25) circle (2pt) ;
        \draw (0,2.3) node[anchor=center]{$Y_1$};
        \draw (-1.832, -1.3) node[anchor=center]{$Y_3$};
        \draw (1.832, -1.3) node[anchor=center]{$Y_2$};
\end{tikzpicture}
    \end{center}
    \caption{Geometric and tropical picture at $t_i = t_j = t_k =0$ in $X[n]$.}
    \label{124}
\end{figure}

\emph{Three basis directions vanish.} When $t_i=t_j=t_k=0$, where $i<j<k$, and all other basis directions are non-zero, then in the locus $\pi^*((Y_1\cap Y_2)^\circ)$, we see exactly two expanded components, which are both of mixed type. Note that, more generally in any fibre of $X[n]$, all expanded components in the $\pi^*((Y_1\cap Y_2)^\circ)$ locus are of mixed type. This is because, in any fibre of $X[n]$, we have that $\Delta_1^{(l)} = \Delta_2^{(n+1-l)}$ in the $\pi^*((Y_1\cap Y_2)^\circ)$ locus for all $l$.

In the example given here, the two bubbles in the $\pi^*((Y_1\cap Y_2)^\circ)$ locus can be described as follows. The bubble which intersects $Y_1$ non-trivially is given by $\Delta_1^{(i)} = \ldots = \Delta_1^{(j-1)}$. By the above, each of these $\Delta_1$-components is equivalent to a $\Delta_2$-component in the $\pi^*((Y_1\cap Y_2)^\circ)$ locus, so this bubble is equivalently given by $\Delta_2^{(n+1-j)} = \ldots = \Delta_2^{(n+1-i)}$. The second bubble in this locus, which intersects $Y_2$ non-trivially, is given by
\[
\Delta_1^{(j)} = \ldots = \Delta_1^{(k-1)} = \Delta_2^{(n+2-k)} = \ldots = \Delta_2^{(n+1-j)}.
\]
 There is a single bubble expanded out in the $\pi^*(Y_1\cap Y_2\cap Y_3)$ locus. This is a $\PP^1\times\PP^1$, given by the meeting of the $\Delta_1^{(i)} = \ldots = \Delta_1^{(j-1)}$ and $\Delta_2^{(n+2-k)} = \ldots = \Delta_2^{(n+1-j)}$ components. Finally, in the $\pi^*((Y_1\cap Y_3)^\circ)$ locus we see exactly two bubbles given by the two distinct expanded $\Delta_1$-components and in the $\pi^*((Y_2\cap Y_3)^\circ)$ locus we see also two bubbles given by the two distinct expanded $\Delta_2$-components. This can be seen in Figure \ref{124}. The intersection of the two edges in the interior of the triangle in the tropical picture creates a new vertex, corresponding to the new bubble in the $\pi^*(Y_1\cap Y_2\cap Y_3)$ locus. The other modified special fibres in $X[n]$ can be described similarly.

\medskip
Now, we note that there is a natural inclusion
\begin{align}\label{standard embedding1}
    C[n] &\longhookrightarrow C[n+1] \\
    (t_1, \ldots, t_{n+1})
    &\longmapsto (t_1, \ldots, t_{n+1}, 1), \nonumber
\end{align}
which, in turn, induces a natural inclusion
\[
X[n] \longhookrightarrow X[n+1].
\]
Under these inclusions, we may consider the space $X[n]$ as a locus in a larger space $X[n+k]$ where all $t_i \neq 0$ for $i>n+1$.

\subsubsection*{The group action.}
We may define a group action on $X[n]$ very similarly to \cite{GHH}. Let $G\subset \SL(n+1)$ be the maximal diagonal torus. We have $\GG_m^n\cong G\subset\GG_m^{n+1}$, where we can view an element of $G$ as an $(n+1)$-tuple $(\sigma_1,\dots,\sigma_{n+1})$ such that $\prod_i \sigma_i = 1$. This acts naturally on $\AA^{n+1}$, which induces an action on $C[n]$. The isomorphism $\GG_m^n\cong G$ is given by
\[
(\tau_1, \ldots, \tau_{n})\longrightarrow (\tau_1, \tau_1^{-1}\tau_2, \ldots, \tau_{n-1}^{-1}\tau_{n}, \tau_n^{-1}).
\]
We shall use the notation $(\tau_1, \ldots, \tau_{n})$ to describe elements of $G$ throughout this work.

\begin{proposition}\label{group proposition}
There is a unique $G$-action on $X[n]$ such that $X[n]\to X\times_{\AA^1}\AA^{n+1}$ is equivariant with respect to the natural action of $G$ on $\AA^{n+1}$.

\medskip
This action is the restriction of the action on $(X\times_{\AA^1}\AA^{n+1})\times (\PP^1)^{2n}$, which is trivial on $X$, acts by
        \begin{align*}
            t_1 &\longmapsto \tau_1^{-1}t_1  \\
            t_k &\longmapsto \tau_k^{-1}\tau_{k-1} t_k \\
            t_{n+1} &\longmapsto \tau_{n} t_{n+1}
        \end{align*}
        on the basis directions, and acts by
        \begin{align*}
            (x_0^{(k)}:x_1^{(k)}) &\longmapsto (\tau_k x_0^{(k)}: x_1^{(k)})  \\
            (y_0^{(k)}:y_1^{(k)}) &\longmapsto (y_0^{(k)}: \tau_{n+1-k} y_1^{(k)}).
        \end{align*}
        on the $\Delta$-components.
\end{proposition}

\begin{proof}
    This follows immediately from \cite{GHH}.
\end{proof}

Note that the group action on the $(y_0^{(k)}:y_1^{(k)})$ coordinates follows immediately from the fact that $\Delta_1^{(k)} = \Delta_2^{(n+1-k)}$ in the $\pi^*((Y_1\cap Y_2)^\circ)$ locus. Given the equations of the blow-ups above, there is no other possible choice of action such that the map $\pi:X[n] \to X\times_{\AA^1}\AA^{n+1}$ is $G$-equivariant (the equations must be invariant under group action). Note also that the natural inclusions
\[
X[n] \longhookrightarrow X[n+k],
\]
we described in the previous section are equivariant under the group action. 

\begin{lemma}\label{G-invariant section of base}
    We have the isomorphism
    \[
    H^0(C[n],\cO_{C[n]})^G \cong k[t],
    \]
    where $H^0(C[n],\cO_{C[n]})^G$ denotes the space of $G$-invariant sections of $H^0(C[n],\cO_{C[n]})$.
\end{lemma}

\begin{proof}
    This is immediate from the above description of the group action.
\end{proof}

\begin{remark}
    We abuse notation slightly by referring to the group acting on $X[n]$ by $G$, instead of $G[n]$. It should always be clear from the context what group $G$ is meant.
\end{remark}

\subsection{Embedding into product of projective bundles}\label{prod proj buns}

In this section, we show how $X[n]$ can be embedded into a fibre product of projective bundles, which locally corresponds to the embedding in $(X\times_{\AA^1}\AA^{n+1})\times (\PP^1)^{2n}$. The $G$-action on $X[n]$ may be expressed as a restriction of a global action on this product of projective bundles. We will then be able to define a $G$-linearised ample line bundle $\cL$ on $X[n]$ by taking the tautological bundle of this fibre product of projective bundles. From this line bundle we will then construct a second line bundle $\cM$ on the relative Hilbert scheme of $m$ points $H^m_{[n]} \coloneqq \Hilb^m(X[n]/C[n])$ with an induced $G$-linearisation.

\medskip
Let $\pr_1$ and $\pr_2$ be the projections of $X \times_{\AA^1} \AA^{n+1}$ to $X$ and $\AA^{n+1}$ respectively. Similarly to \cite{GHH}, we define vector bundles
\begin{align*}
    \cF_{1}^{(k)} &= \pr_1^* \cO_X(-Y_{1}) \oplus \pr_2^* \cO_{\AA^{n+1}}(-V(t_k)) \\
    \cF_{2}^{(k)} &= \pr_1^* \cO_X(-Y_{2}) \oplus \pr_2^* \cO_{\AA^{n+1}}(-V(t_{n+2-k}))
\end{align*}
on $X\times_{\AA^1}\AA^{n+1}$.

\begin{lemma}
    There is an embedding
    \[
    X[n] \longhookrightarrow \prod_{i,j}\PP( \cF_i^{(j)}),
    \]
    where the product of projective bundles $\prod_{i,j}\PP( \cF_i^{(j)})$ is constructed as a fibre product over $X\times_{\AA^1}\AA^{n+1}$.
\end{lemma}

\begin{proof}
Let $\cI_1^{(k)},\cI_2^{(k)}$ be the ideal sheaves corresponding to each blow-up we perform; for example $\cI_1^{(1)}$ is the ideal sheaf of $Y_{1} \times_{\AA^1} V(t_1)$ on $X\times_{\AA^1}\AA^{n+1}$. Then $\cI_2^{(1)}$ is the ideal sheaf of $b_{(1,0)}^*(Y_{2} \times_{\AA^1} V(t_{n+1}))$ on $X_{(1,0)}$, and so on for $\cI_j^{(k)}$.

As we will explain below, we then have, for each of the vector bundles $\cF_{1}^{(k_1)}$ and $\cF_{2}^{(k_2)}$, the embeddings
\begin{align*}
    X_{(k_1,k_2)} &\longhookrightarrow \PP(b_{(k_1-1,k_2)}^* \cF_1^{(k_1)}), \\
    X_{(k_1,k_2)} &\longhookrightarrow \PP(b_{(k_1,k_2-1)}^* \cF_2^{(k_2)}),
\end{align*}
where $b_{(0,0)}$ is understood to be just the identity map on $X_{(0,0)} = X\times_{\AA^1}\AA^{n+1}$. Indeed, the scheme $X_{(k_1,k_2)}$ embeds into the projectivisations of the ideals of these blow-ups $\PP(\cI_1^{(k_1)})$ and $ \PP(\cI_2^{(k_2)})$. For a reference on projectivisations of ideals see \cite{EH}. There is a surjection 
\[
b_{(k_1-1,k_2)}^*\cF_1^{(k_1)} \longrightarrow \cI_1^{(k_1)} \text{ given by } \begin{pmatrix}
b_{(k_1-1,k_2)}^* x\\
t_{k_1}
\end{pmatrix},
\]
where $x$ is a defining equation of the locus to be blown up projected forward to $X$, i.e.\ it is the defining equation of $Y_{1}$. Similarly, there is a surjection
\begin{align*}
    b_{(k_1,k_2-1)}^*\cF_2^{(k_2)} &\longrightarrow \cI_2^{(k_2)}.
\end{align*}
From this, we deduce that there are embeddings 
\begin{align*}
    \PP(\cI_1^{(k_1)}) &\longhookrightarrow \PP(b_{(k_1-1,k_2)}^* \cF_1^{(k_1)}), \\
    \PP(\cI_2^{(k_2)}) &\longhookrightarrow \PP(b_{(k_1,k_2-1)}^* \cF_2^{(k_2)}).
\end{align*}
Hence we have embeddings 
\begin{figure}[H]
    \centering
    \begin{tikzcd}
        &X_{(k_1,k_2)} \arrow[ld, hook]  \arrow[rd, hook]\\
        \PP(b_{(k_1-1,k_2)}^* \cF_1^{(k_1)})  &&\PP(b_{(k_1,k_2-1)}^* \cF_2^{(k_2)}).
    \end{tikzcd}
    \label{embeddings}
\end{figure}
Now, similarly to \cite{GHH}, we can embed $X[n] = X_{(n,n)}$ into $\prod_{i,j}\PP( \cF_i^{(j)})$, which is to be understood as the fibre product over $X\times_{\AA^1} \AA^{n+1}$. This can be seen by iteration on $i,j$ in the following way. The simplest case is $X_{(1,0)} \hookrightarrow \PP(b_{(0,0)}^* \cF_1^{(1)}) = \PP( \cF_1^{(1)})$, which is obvious. Then for $X_{(1,1)}$, we have the following commutative diagram
\begin{figure}[H]
    \centering
    \begin{tikzcd}
X_{(1,1)}\subset b_{(1,0)}^* \PP(\cF_2^{(1)}) \arrow[r] \arrow[d] &X_{(1,0)} \arrow[d, "b_{(1,0)}"]\subset \PP(\cF_1^{(1)}) \\
  \PP(\cF_2^{(1)}) \arrow[r] &X\times_{\AA^1} \AA^{n+1}
\end{tikzcd}
    \label{fig:my_label}
\end{figure}
\noindent (recall $b_{(1,0)}^* \PP(\cF_2^{(1)})$ is a vector bundle over $X_{(1,0)}$ and $\PP(\cF_2^{(1)})$ is a vector bundle over $X\times_{\AA^1} \AA^{n+1}$, giving us the horizontal maps). By the universal property of fibre products, there is a unique map $X_{(1,1)} \to \PP(\cF_1^{(1)})\times \PP(\cF_2^{(1)})$. But by universal property of the pullback there is also a unique map $\PP(\cF_1^{(1)})\times \PP(\cF_2^{(1)}) \to b_{(1,0)}^* \PP(\cF_2^{(1)})$, hence the embedding $X_{(1,1)} \hookrightarrow b_{(1,0)}^*\PP(\cF_2^{(1)})$ factors through $\PP(\cF_1^{(1)})\times \PP(\cF_2^{(1)})$. Since the composition of the two maps is injective, the first map, i.e.\ $X_{(1,1)} \to \PP(\cF_1^{(1)})\times \PP(\cF_2^{(1)})$, must be injective and the image in $\PP(\cF_1^{(1)})\times \PP(\cF_2^{(1)})$ is closed by properness. We can then iterate this argument until we obtain the embedding $X_{(n,n)}\hookrightarrow \prod_{i,j}\PP( \cF_i^{(j)})$.
\end{proof}

\medskip
The $G$-action is a restriction of the torus action on $\prod_{i,j}\PP( \cF_i^{(j)})$, described étale locally in Proposition \ref{group proposition}.

\subsubsection*{Linearisations.}

The following lemma gives a method to construct all the linearised line bundles we will need to vary the GIT stability condition.

\begin{lemma}\label{line bun}
    There exists a $G$-linearised ample line bundle $\cL$ on $X[n]$ such that locally the lifts to this line bundle of the $G$-action on each $\PP^1$ corresponding to a $\Delta_1^{(k)}$ and on each $\PP^1$ corresponding to a $\Delta_2^{(n+1-k)}$ are given by
    \begin{align}
        \label{1lift1}
        (x_0^{(k)};x_1^{(k)}) &\longmapsto (\tau_k^{a_k} x_0^{(k)} ; \tau_k^{-b_k} x_1^{(k)}) \\
        \label{1lift2}
        (y_0^{(n+1-k)};y_1^{(n+1-k)}) &\longmapsto (\tau_k^{-c_k} y_0^{(n+1-k)} ; \tau_k^{d_k} y_1^{(n+1-k)}) 
    \end{align}
    for any choice of positive integers $a_k,b_k,c_k,d_k$.
\end{lemma}

\begin{proof}
    Similarly to the proof of Lemma 1.18 in \cite{GHH}, we see that each locally free sheaf $\cF_i^{(k)}$ on $X\times_{\AA^1} \AA^{n+1}$ has a canonical $G$-linearisation. There is an induced $G$-action on the projective product $\prod_{i,k}\PP( \cF_i^{(k)})$, which is equivariant under the embedding
    \[
    X[n] \longhookrightarrow \prod_{i,k}\PP( \cF_i^{(k)}).
    \]
    The $G$-action on each $\PP( \cF_i^{(k)})$ lifts to a $G$-action on the corresponding vector bundle, which gives us a canonical linearisation of the line bundle $\cO_{\PP( \cF_i^{(k)})}(1)$. Locally, the actions on $\cO_{\PP( \cF_1^{(k)})}(1)$ and $\cO_{\PP( \cF_2^{(k)})}(1)$ are given respectively by
    \begin{equation}
        (x_0^{(k)} ; \tau_k^{-1} x_1^{(k)}) \quad \text{and} \quad
        ( y_0^{(k)} ; \tau_{n+1-k} y_1^{(k)}).\nonumber
    \end{equation}
    We therefore may define the lifts \eqref{1lift1} and \eqref{1lift2} on the line bundles $\cO_{\PP( \cF_1^{(k)})}(a_k+b_k)$ and $\cO_{\PP( \cF_{2}^{(n+1-k)})}(c_k+d_k)$ respectively. We then pull back each $\cO_{\PP( \cF_1^{(k)})}(a_k+b_k)$ and $\cO_{\PP( \cF_{2}^{(k)})}(c_{n+1-k}+d_{n+1-k})$ to $\prod_{i,k}\PP( \cF_i^{(k)})$ and form their tensor product to obtain a $G$-linearised line bundle, which we denote by $\cL$.
\end{proof}

Each such line bundle $\cL$ which can be constructed in this way will induce a $G$-linearised line bundle $\cM$ on $H^m_{[n]}$. This, in turn, will yield a GIT stability condition on $H^m_{[n]}$.

\section{GIT stability}\label{GIT stability}

In this section, we set up some results analogous to those of \cite{GHH} to describe various GIT stability conditions on the scheme $X[n]$ with respect to the possible choices of $G$-linearised line bundles described in the previous section. In particular, we show that these stability conditions do not depend on the scheme structure of the length $m$ zero-dimensional subschemes, but instead can be reduced to combinatorial criteria on configurations of $n$ points.

\subsection{Hilbert-Mumford criterion}

In this section, we shall recall the definition of Hilbert-Mumford invariants and give a numerical criterion for stability and semi-stability in terms of these invariants.

\medskip
Let $H$ be a reductive group acting on a scheme $S$, which is proper over an algebraically closed field $k$. Let $L$ be a $H$-linearised ample line bundle. Then a \emph{1-parameter subgroup} of $H$ (denoted 1-PS for convenience) is defined to be a homomorphism
\[
\lambda\colon \GG_m \to H.
\]
Now let $P$ be any point in $S$. For $\tau\in \GG_m$, we denote by $P_0$ the limit of $\tau P$ as $\tau$ tends towards 0 if such a limit exists. Then let $\mu^L(\lambda,P)$ be the negative of the weight of the $\GG_m$-action on the fibre $L(P_0)$. We call $\mu^L(\lambda,P)$ a \emph{Hilbert-Mumford invariant}.

\medskip
In our case we will want to think of $H$ as being our group $G$, of $S$ as being the relative Hilbert scheme of points $H^m_{[n]}$ and of $L$ as being the line bundle $\cM$ on $H^m_{[n]}$, which we define in the next section. A 1-parameter subgroup of $G$ will be given by a map
\[
\lambda\colon \GG_m \to G, \quad \tau \mapsto ( \tau^{s_1}, \ldots, \tau^{s_{n}}),
\]
where $(s_1,\ldots,s_{n})\in \ZZ^{n}$. The following result will allow us to use these invariants to determine stability and semi-stability in our GIT constructions. It is a relative version of the Hilbert-Mumford criterion (see Mumford, Fogarty and Kirwan \cite{MFK}) proven by Gulbrandsen, Halle and Hulek in \cite{HM}.

\begin{theorem}
Let $k$ be an algebraically closed field and $f\colon S \to B$ a projective morphism of $k$-schemes. Assume $B=\Spec A$ is noetherian and $B$ is of finite type over $k$. Let $H$ be an affine, linearly reductive group over $k$ acting on $S$ and $B$ such that $f$ is equivariant and let $L$ be an ample $H$-linearised line bundle on $S$. Suppose $P\in S$ is a closed point. Then $P$ is stable (or semistable) if and only if $\mu^L(\lambda, P)> 0 $ (or $ \geq 0$) for every non-trivial 1-PS $\lambda\colon \GG_m \to H$.

\end{theorem}

\subsection{Action of 1-parameter subgroup}\label{1-PS}

\subsubsection*{Existence of limits under action of a 1-PS.}
Let $P$ be any point in $X[n]$ and let $p_n \colon X[n] \to C[n]$ be the projection to the base. As stated in \cite{GHH}, the limit $P_0$ of $P$ under a 1-PS as defined above exists if and only if its projection onto the base, $p_n(P)\in C[n]$, has a limit. The $G$-action on the base is a pullback of the action on $\AA^{n+1}$ and the corresponding action of a 1-PS is
\begin{align*}
    t_1 &\longmapsto \tau^{-s_1}t_1, \\
    t_k &\longmapsto \tau^{s_{k-1}-s_k}t_k, \quad \textrm{for } 1<k\leq n, \\
    t_{n+1} &\longmapsto \tau^{s_n}t_{n+1}.
\end{align*}
The projection $p_n(P)$ of the point $P$ to the base has a limit as $\tau$ tends to zero if and only if each power of $\tau$ in the action is nonnegative on the nonzero basis directions $t_i$, i.e.\ if and only if
\begin{align}\label{boundedness}
    &0\geq s_1\geq \ldots \geq s_{n+1} \geq 0,
\end{align}
where each inequality from left to right must hold if $t_1, \ldots, t_{n+1}$ is nonzero respectively. Thus we obtain boundedness conditions on the weights $s_i$ dependent on where $P$ lies over the base. In particular, when $t_i\neq 0$ for all $i$, this implies that $s_i=0$ for all $i$, so the 1-PS are trivial and all points are trivially semistable.

\subsubsection*{Lifts of 1-PS action to the line bundle.}
Let $\cL$ be a line bundle as described in Lemma \ref{line bun}. Assume that locally the lifts to $\cL$ of the $G$-action on each $\PP^1$ corresponding to a $\Delta_1^{(k)}$ and on each $\PP^1$ corresponding to a $\Delta_2^{(n+1-k)}$ are given by
    \begin{align*}
        (x_0^{(k)};x_1^{(k)}) &\longmapsto (\tau_k^{a_k} x_0^{(k)} ; \tau_k^{-b_k} x_1^{(k)}) \\
        (y_0^{(n+1-k)};y_1^{(n+1-k)}) &\longmapsto (\tau_k^{-c_k} y_0^{(n+1-k)} ; \tau_k^{d_k} y_1^{(n+1-k)}) 
    \end{align*}
for some choice of positive integers $a_k,b_k,c_k,d_k$. Then the corresponding lifts of the 1-PS action to $\cL$ are given by
\begin{align*}
    (x_0^{(k)};x_1^{(k)}) &\longmapsto (\tau^{a_ks_k} x_0^{(k)}: \tau^{-b_ks_k}x_1^{(k)}) , \\
    (y_0^{(n+1-k)};y_1^{(n+1-k)}) &\longmapsto (\tau^{-c_ks_{n+1-k}} y_0^{(k)}: \tau^{d_k s_{n+1-k}}y_1^{(k)}).
\end{align*}
We will see in the next section that the Hilbert-Mumford invariants that interest us, which are the invariants relating to 1-PS subgroups of the induced action of $G$ on $H^m_{[n]}$ with associated line bundle $\cM$, can be calculated by simply adding the invariants of the points (with multiplicity) in $X[n]$ which make up the support of an element of $H^m_{[n]}$. Indeed, we will see that it is possible for this purpose to think of an element of $H^m_{[n]}$ as just a union of points with multiplicity in $X[n]$ and forget about its scheme structure.

\subsection{Bounded and combinatorial weights}\label{bounded and comb weights}
In this section, we explain the relation between what \cite{GHH} call the bounded and combinatorial weights of the Hilbert-Mumford invariants.

\medskip
Keeping the notation as consistent as possible with \cite{GHH}, let
\[
Z^m_{[n]}\subset H^m_{[n]}\times_{C[n]} X[n]
\]
be the universal family, with first and second projections $p$ and $q$. The line bundle
\[
\cM_l\coloneqq \det p_*(q^*\cL^{\otimes l}|_{Z^m_{[n]}})
\]
is relatively ample when $l\gg0$ and is $G$-linearised, exactly as in Section 2.2.1 of \cite{GHH}.

\subsubsection*{Relationship between bounded and combinatorial weights.} The following lemmas describe how the Hilbert-Mumford invariant can be decomposed into a sum of invariants.

\begin{lemma}
    Given a point $[Z]\in H^m_{[n]}$ and a 1-PS $\lambda_s$ given by $(s_1,\ldots,s_{n})\in \ZZ^{n}$, denote the limit of $\lambda_s(\tau)\cdot Z$ as $\tau$ tends to zero by $Z_0$. The Hilbert-Mumford invariant can be decomposed into a sum
    \[
    \mu^{\cM_l}(Z,\lambda_s) = \mu_b^{\cM_1}(Z,\lambda_s) + l \cdot \mu_c^{\cM_1}(Z,\lambda_s)
    \]
    of the \emph{bounded weight} $\mu_b^{\cM_l}(Z,\lambda_s)$, coming from the scheme structure of $Z_0$, and the \emph{combinatorial weight} $\mu_c^{\cM_l} (Z,\lambda_s)$, coming from the weights of the 1-PS action on $\cL$.
\end{lemma}

\begin{proof}
 The Hilbert-Mumford invariant $\mu^{\cM_l}(Z,\lambda_s)$ is given by the negative of the weight of the $\GG_m$-action on the line bundle $\cM_l$ at the point $Z_0$. At the point $Z_0$, the line bundle $\cM_l$ is given by $\det (H^0(\cO_{Z_0} \otimes \cL^{\otimes l}))$. We can write $Z_0$ as a union of length $n_P$ zero-dimensional subschemes $\bigcup_P Z_{0,P}$ supported at points $P$. Let $\cL^{\otimes l}(P)$ denote the fibre of $\cL^{\otimes l}$ at $P$. Following \cite{GHH}, there is an isomorphism

\[
H^0(\cO_{Z_0} \otimes \cL^{\otimes l}) \cong \bigoplus_P \bigl( H^0(\cO_{Z_0,P}) \otimes \cL^{\otimes l}(P) \bigl).
\]
Then, by taking determinants, as in \cite{GHH}, we get
\[
\Wedge^m H^0(\cO_{Z_0} \otimes \cL^{\otimes l}) \cong
\Bigl( \Wedge^m H^0(\cO_{Z_0})\Bigl) \otimes \Bigl( \bigotimes_P \cL^{\otimes l n_P}(P) \Bigl).
\]
which allows us to write the invariant $\mu^{\cM_l}(Z,\lambda_s)$ as a sum of its \emph{bounded weight} $\mu_b^{\cM_l}(Z,\lambda_s)$, coming from $\Wedge^m H^0(\cO_{Z_0})$ in the above, and its \emph{combinatorial weight} $\mu_c^{\cM_l} (Z,\lambda_s)$, coming from $\bigotimes_P \cL^{\otimes ln_P}(P)$. It is clear also that
\[
\mu_b^{\cM_l}(Z,\lambda_s) = \mu_b^{\cM_1}(Z,\lambda_s),
\]
since the bounded weight does not depend on the value of $l$, and
\[
\mu_c^{\cM_l}(Z,\lambda_s) = l \cdot \mu_c^{\cM_1}(Z,\lambda_s).
\]
Hence, we have
\[
\mu^{\cM_l}(Z,\lambda_s) = \mu_b^{\cM_1}(Z,\lambda_s) + l \cdot \mu_c^{\cM_1}(Z,\lambda_s).
\]
\end{proof}

Note that, whereas the combinatorial weight depends on the choice of linearised line bundle, the bounded weight does not. Similarly to \cite{GHH}, we can show that the bounded weight, as its name suggests, can be given an upper bound.

\medskip
\emph{Terminology.} In the proofs of the next results, we will say that a point of the support of a subscheme is on a certain \emph{side} of a $\Delta$-component to describe whether it lies on the $(0:1)$ or $(1:0)$ side of the corresponding $\PP^1$.

\medskip
The following result is based on Lemma 2.3 of \cite{GHH}, with some slight modifications to suit our setting.

\begin{lemma}\label{bounded}
Let $\mu_b^{\cM_1}(Z,\lambda_s)$ be the bounded weight of $[Z]\in H_{[n]}^m$ and $s\in \ZZ^{n}$ such that the limit of $\lambda_s(\tau)\cdot Z$ as $\tau$ goes to zero exists. Then
\[
\mu_b^{\cM_1}(Z,\lambda_s) = \sum_{i=1}^{n} b_is_i,
\]
where $|b_i|\leq 2m^2$ for every $i$.
\end{lemma}

\begin{proof}
Let $Z_0$ be the limit point of $Z$ with respect to some $s\in \ZZ^{n}$, where $Z_0$ is supported at points $Q_i\in X[n]$. Since $Z_0$ is a limit point of the action, each $Q_i$ must be a $\GG_m$-fixpoint. And since $Z_{0,Q_i}$ is a finite local scheme, we can work with the local coordinates we set up earlier.

\medskip
Following our previous notation, let $n_{Q_i}$ denote the multiplicity of the scheme $Z_0$ at the point $Q_i$. The coordinate ring of $Z_{0,Q_i}$ is then generated by $n_{Q_i}$ monomials in the variables $x,y,z,t_1,\ldots,t_{n+1}$ and $x_0^{(k)}/x_1^{(k)}$ or $x_1^{(k)}/x_0^{(k)}$ depending on which side of $\Delta_1^{(k)}$ the point $Q_i$ lies, and $y_0^{(k)}/y_1^{(k)}$ or $y_1^{(k)}/y_0^{(k)}$ depending on which side of $\Delta_2^{(k)}$ the point $Q_i$ lies. Note that the coordinate ring of $Z_{0,Q_i}$ will only contain monomials in the variable $x_0^{(k)}/x_1^{(k)}$ or $x_1^{(k)}/x_0^{(k)}$ if $Q_i\in \Delta_1^{(k)}$, and similarly for the variable $y_0^{(k)}/y_1^{(k)}$ or $y_1^{(k)}/y_0^{(k)}$. Moreover, if $Q_i\in (\Delta_1^{(k)})^\circ \cup (\Delta_2^{(n+1-k)})^\circ$, then this means that $s_k =0$ as $Z_0$ is the limit of the 1-PS action. So the weight of the $\GG_m$-action on a $\Delta$-component will be nontrivial only if $Q_i$ lies on the boundary of this component.

\medskip
The weight $b_k$ restricted to the point $Q_i$ is given by adding the multiplicity of $x_0^{(k)}/x_1^{(k)}$ times $s_k$ or that of $x_1^{(k)}/x_0^{(k)}$ times $-s_k$ (depending on which side of $\Delta_1^{(k)}$ the point $Q_i$ lies) in each monomial, plus the multiplicity of $y_0^{(n+1-k)}/y_1^{(n+1-k)}$ times $-s_k$ or that of $y_1^{(n+1-k)}/y_0^{(n+1-k)}$ times $s_k$ (depending on which side of $\Delta_2^{(n+1-k)}$ the point $Q_i$ lies) in each monomial. Each monomial has degree at most $n_{Q_i}$. The parts of $b_k$ coming from the actions on $\Delta_1^{(k)}$ and $\Delta_2^{(n+1-k)}$ therefore both have absolute value at most $m^2$, so $|b_k|\leq 2m^2$.
\end{proof}

Let us discuss now how the bounded weight affects the overall stability condition. The following lemma is immediate from \cite{GHH}, but we recall their proof here for convenience.

\begin{lemma}\label{comb over bound}
    Let $Z$ be a length $m$ zero-dimensional subscheme in a fibre of $X[n]$. Assume that, for all $s\in \ZZ^n$ such that the limit $\lambda_s(\tau)\cdot Z$ as $\tau$ tends to zero exists, the combinatorial weight can be written as
    \[
    \mu_c^{\cM_1}(Z,\lambda_s) = \sum_{i=1}^{n} c_is_i,
    \]
    where $c_is_i\geq 0$ with equality if and only if $s_i=0$. Then $Z$ is stable with respect to the $G$-linearised line bundle $\cM_l$ on $H^m_{[n]}$ for some large enough~$l$.
\end{lemma}

\begin{proof}
    As we have shown that the bounded weight can be expressed as
    \[
    \mu_b^{\cM_1}(Z,\lambda_s) = \sum_{i=1}^{n} b_is_i,
    \]
    where $|b_i|\leq 2m^2$, and recalling that the Hilbert-Mumford invariant can be expressed as
    \[
    \mu^{\cM_l}(Z,\lambda_s) = \mu_b^{\cM_1}(Z,\lambda_s) + l \cdot \mu_c^{\cM_1}(Z,\lambda_s),
    \]
    it is just a matter of choosing a big enough value of $l$ to make the combinatorial weight overpower the bounded weight. This allows us effectively to treat the bounded weight as negligible and ignore it in our computations.
\end{proof}

\begin{remark}
    The assumption of Lemma \ref{comb over bound} does not hold in general for all possible $G$-linearised line bundles on $H^m_{[n]}$.
\end{remark}

\subsubsection*{A criterion for positive combinatorial weights.}

 Let $Z$ be a length $m$ zero-dimensional subscheme in a fibre of $X[n]$. With the following lemmas, we shall establish that if there is at least one point of the support of $Z$ in the union $(\Delta_1^{(k)})^\circ \cup (\Delta_2^{(n+1-k)})^\circ$ for every $k$ (where these $\Delta$-components are not necessarily expanded out), then there exists a GIT stability condition which makes $Z$ stable. We start by showing that for such a subscheme $Z$ there exists a $G$-linearised line bundle $\cM$ on $H^m_{[n]}$ such that the corresponding combinatorial weight will be strictly positive. We will then use Lemma \ref{comb over bound} to show that $Z$ is stable in the corresponding GIT stability.

 \begin{remark}
     Note, here, that such a $Z$ will not necessarily be smoothly supported, nor will every point of the support of $Z$ necessarily be contained in a $\Delta$-component.
 \end{remark}
 
\begin{lemma}\label{line bun for pos comb weight}
    Let $Z$ be in a fibre of $X[n]$ as above. If there is at least one point of the support of $Z$ in the union $(\Delta_1^{(k)})^\circ \cup (\Delta_2^{(n+1-k)})^\circ$ for every $k$, then there exists a $G$-linearised line bundle on $H^m_{[n]}$ with respect to which the combinatorial weight of $Z$ is strictly positive for every nontrivial 1-PS $\lambda_s$ such that the limit of $\lambda_s(\tau)\cdot Z$ as $\tau$ tends to zero exists.
\end{lemma}

\begin{proof}
    We will construct a $G$-linearised line bundle $\cL$ on $X[n]$ as in Lemma \ref{line bun}, by specifying lifts of the $G$-action on each $\PP(\cF_1^{(k)})$ and $\PP(\cF_2^{(n+1-k)})$ to line bundles $\cO_{\PP( \cF_1^{(k)})}(a_k+b_k)$ and $\cO_{\PP( \cF_{2}^{(n+1-k)})}(c_k+d_k)$ for some chosen values $a_k,b_k,c_k,d_k\in \ZZ_{\geq 0}$.

    \medskip
    Let $k\in \{ 1, \dots, n \}$. If there is some point of the support of $Z$, denoted $P$, in $(\Delta_1^{(k)})^\circ \subseteq \pi^*(Y_1\cap Y_3)$, and if $m'$ points of the support lie on the $(1:0)$ side of $\Delta_1^{(k)}$, then we will want the lift of the $G$-action on $\PP(\cF_1^{(k)})$ to $\cO_{\PP( \cF_1^{(k)})}(a_k+b_k)$ to be locally given by
    \[
    (x_0^{(k)}; x_1^{(k)}) \longmapsto (\tau_k^{m(m-m')} x_0^{(k)}; \tau_k^{-m(m'+1)} x_1^{(k)})\in \AA^2.
    \]
    This lift is therefore defined on $\cO_{\PP(\cF_1^{(k)})}(m^2+m)$, i.e.\ we have chosen $a_k= m(m-m')$ and $b_k = m(m'+1)$. We will then choose $c_k = 0$ and $d_k = 1$, so that the action on $\PP( \cF_{2}^{(n+1-k)})$ lifts to $\cO_{\PP( \cF_{2}^{(n+1-k)})}(1)$ and it is locally given by
    \[
    (y_0^{(n+1-k)}; y_1^{(n+1-k)}) \longmapsto ( y_0^{(n+1-k)}; \tau_k y_1^{(n+1-k)})\in \AA^2.
    \]
    If there is no point of the support of $Z$ in $\Delta_1^{(k)}$, we set the lift of the $G$-action on $\PP(\cF_1^{(k)})$ to $\cO_{\PP( \cF_1^{(k)})}(1)$ to be locally given by
    \[
    (x_0^{(k)}; x_1^{(k)}) \longmapsto (\tau_k x_0^{(k)}; x_1^{(k)})\in \AA^2,
    \]
    i.e.\ we have chosen $a_k=1$ and $b_k=0$. In this case there must be at least one point of the support in $(\Delta_2^{n+1-k})^\circ$. Let $m''$ be the number of points of the support on the $(1:0)$ side of $\Delta_2^{(n+1-k)}$. We then set $c_k = m(m-m'')$ and $d_k = m(m''+1)$, i.e.\ we have a lift of the $G$-action on $\PP( \cF_{2}^{(n+1-k)})$ to $\cO_{\PP( \cF_{2}^{(n+1-k)})}(m^2+m)$, locally given by
    \[
    (y_0^{(n+1-k)}; y_1^{(n+1-k)}) \longmapsto (\tau_k^{-m(m-m'')} y_0^{(n+1-k)}; \tau_k^{-m(m''+1)} y_1^{(n+1-k)})\in \AA^2.
    \]
    Repeating this process over all $k\in \{ 1, \dots, n \}$ will give us a description of $\cL$ and we may form the $G$-linearised line bundle $\cM$ from this line bundle in the way described at the start of this section. For more details on why this yields a positive combinatorial weight, see the proof of the following lemma. Note that this is not the only GIT stability condition for which $Z$ is stable.
\end{proof}

\begin{lemma}\label{pos terms in comb sum}
    Let $Z$ be as in the statement of Lemma \ref{line bun for pos comb weight} and let $\cM$ be a $G$-linearised line bundle constructed as in the proof of Lemma \ref{line bun for pos comb weight}. Then, for any $s\in \ZZ^n$, the combinatorial weight can be written
    \[
    \mu_c^{\cM}(Z,\lambda_s) = \sum_{i=1}^{n} c_is_i,
    \]
    where $c_is_i \geq 0$ with equality if and only if $s_i=0$.
\end{lemma}

\begin{proof}
    It is clear that the combinatorial weight may be written as a sum
    \[
    \mu_c^{\cM}(Z,\lambda_s) = \sum_{i=1}^{n} c_is_i.
    \]
    Now, let us take any $k\in \{ 1, \dots, n \}$. First, let us assume that there is at least one point of the support in $(\Delta_1^{(k)})^\circ \subseteq \pi^*(Y_1\cap Y_3)$ and denote by $m'$ the number of points of the support on the $(1:0)$ side of $\Delta_1^{(k)}$. Then, if $s_k>0$,
    \begin{align*}
        &c_ks_k \geq (-m'm(m-m') + (m-m')m(m'+1))s_k -(m')s_k = (m^2-m')s_k \geq 0.
    \end{align*}
    Here, $m^2$ corresponds to the weight coming from $\PP( \cF_{1}^{(k)})$ and $m'$ corresponds to the weight coming from $\PP( \cF_{2}^{(n+1-k)})$. The value $m'$ arises from the fact that there are at most $m'$ points of the support on the $(0:1)$ side of $\Delta_2^{(n+1-k)}$. And since $m'$ can be at most $m-1$, we have that $m^2-m'>0$.

    Now, if $s_k<0$, then
    \[
    c_ks_k \geq (-(m'+1)m(m-m') +(m-m'-1)m(m'+1))s_k +0 = -(m'+1)ms_k \geq 0,
    \]
    where again the two terms correspond to the weights coming from $\PP( \cF_{1}^{(k)})$ and $\PP( \cF_{2}^{(n+1-k)})$. As $(m'+1)m>0$, this gives the desired answer.

    Finally, if there is no point of the support in $(\Delta_1^{(k)})^\circ \subseteq \pi^*(Y_1\cap Y_3)$, we can make a very similar argument, as the weight coming from $\PP( \cF_{1}^{(k)})$ is overpowered by the weight coming from $\PP( \cF_{2}^{(n+1-k)})$ in the line bundle $\cM$ we set up.
\end{proof}

\subsection{Semistable locus and GIT quotient}\label{GIT quotients}

\begin{lemma}\label{exists stable}
    Let $Z$ be as in the statement of Lemma \ref{line bun for pos comb weight}. Then there exists a GIT stability condition on $H^m_{[n]}$ which makes $Z$ stable.
\end{lemma}

\begin{proof}
    This follows from Lemmas \ref{comb over bound} and \ref{pos terms in comb sum}. Indeed, by Lemma \ref{comb over bound}, if the combinatorial weight can be written in the form
    \[
    \mu_c^{\cM}(Z,\lambda_s) = \sum_{i=1}^{n} c_is_i,
    \]
    where $c_is_i \geq 0$ with equality if and only if $s_i=0$, then we may choose a high enough tensor power $l$ of $\cM$ such that $Z$ is stable if and only if the combinatorial weight is strictly positive. But this condition is satisfied by Lemma \ref{pos terms in comb sum}.
\end{proof}

\begin{lemma}\label{unstable}
    Let $Z$ be a length $m$ zero-dimensional subscheme in a fibre of $X[n]$, such that no point of the support is contained in the union $(\Delta_1^{(k)})^\circ \cup (\Delta_2^{(n+1-k)})^\circ$ for some $k$ (these components may be expanded or not in the fibre). Then there exists no GIT stability condition on $H^m_{[n]}$ with respect to the group $G$ which makes $Z$ stable.
\end{lemma}

\begin{proof}
    Let us choose an arbitrary $G$-linearised line bundle $\cM$, not necessarily constructed as above, with respect to which $Z$ has Hilbert-Mumford invariant
    \[
    \mu^{\cM}(Z,\lambda_s) = \sum_{i=1}^{n} a_is_i,
    \]
    for some $s\in \ZZ^n$ such that the limit of $\lambda_s(\tau)\cdot Z$ as $\tau$ tends to zero exists.
    Either $a_k = 0$, in which case $Z$ cannot be stable (it will at best be semistable) with respect to the stability condition given by the chosen linearisation, or $a_k\neq 0$.
    
    \medskip
    If $\Delta_1^{(k)}$ and $\Delta_2^{(n+1-k)}$ are expanded out in the fibre, then $s_k$ is not bounded above or below by $0$ or by any weights acting nontrivially outside of these components. Moreover, as no point of the support of $Z$ are contained in $(\Delta_1^{(k)})^\circ \cup (\Delta_2^{(n+1-k)})^\circ$, we know that $\tau^{s_k}$ acts trivially on all points of the support. The integer $a_k$ is therefore independent of the value of $s_k$; different values of $s_k$ will not change $a_k$. If $a_k>0$, we may choose $s_k$ to be negative with large enough absolute value to destabilise $Z$. Similarly, if $a_k<0$, we may choose $s_k$ to be positive and large enough to destabilise~$Z$.

    \medskip
    Finally, if $\Delta_1^{(k)}$ and $\Delta_2^{(n+1-k)}$ are not expanded out in the fibre, either $t_l\neq 0$ for $l \geq k$ or $t_l\neq 0$ for $l \leq k$. If $t_l\neq 0$ for $l \geq k$, then $\Delta_1^{(k)} = Y_1$ and $\Delta_2^{(n+1-k)} = Y_1 \cup Y_3$. All points of the support of $Z$ must therefore be on the $(1:0)$ side of $\Delta_1^{(k)}$ and on the  $(0:1)$ side of $\Delta_2^{(n+1-k)}$, which implies that $a_k<0$. But by the condition \eqref{boundedness}, we have $s_k\geq 0$, and we can therefore choose $s_k$ large enough to destabilise $Z$. A very similar argument can be made if instead $t_l\neq 0$ for $l \leq k$.
\end{proof}

\begin{theorem}\label{stability theorem}
    Let $Z$ be a length $m$ zero-dimensional subscheme in a fibre of $X[n]$. Then there exists a GIT stability condition on $H^m_{[n]}$ which makes $Z$ stable if and only if there is at least one point of the support of $Z$ in $(\Delta_1^{(k)})^\circ \cup (\Delta_2^{(n+1-k)})^\circ$ for every $k$.
\end{theorem}

\begin{proof}
    This follows directly from Lemmas \ref{exists stable} and \ref{unstable}.
\end{proof}

We can now describe the GIT quotients resulting from these constructions. Let
\[
A[n] \coloneqq H^0(C[n],\cO_{C[n]}).
\]
Then we recall from Lemma \ref{G-invariant section of base}, the isomorphism
\[
H^0(C[n],\cO_{C[n]})^G\cong k[t].
\]
For all choices of linearised line bundle described in the above, the GIT quotient on the base therefore behaves as follows
\[
C[n]/\!/G =\Spec A[n]/\!/G = \Spec (A[n]^{G})  \cong \AA^1.
\]
Now let us denote by $H^{m,s}_{[n],\cM}$ the locus of GIT stable subschemes in $H^m_{[n]}$ with respect to the stability condition determined by one of the choices of $G$-linearised line bundle $\cM$ as constructed in Section \ref{bounded and comb weights} and let
\[
I^m_{[n],\cM} \coloneqq H^{m,s}_{[n],\cM}/\!/G
\]
denote the corresponding GIT quotient.

\begin{theorem}
The GIT quotients $I^m_{[n],\cM}$ thus constructed are projective over
\[
\Spec (A[n]^{G}) \cong \AA^1.
\]
\end{theorem}

\begin{proof}
This result follows directly from the relative Hilbert-Mumford criterion of \cite{HM}.
\end{proof}

\section{Stack perspective}\label{stack construction}

In this section, we generalise the scheme construction of Section \ref{Second_constr} and define the analogous stack of expansions and its family $\mathfrak{X}\to \mathfrak{C}$. As mentioned before, we impose additional equivalences in the stack, which have the effect of setting any two fibres with the same expanded components to be equivalent. We examine the loci of GIT stable points again on this stack and discuss their relation with stability conditions of Li and Wu and of Maulik and Ranganathan (\cite{LW}, \cite{MR}). Finally, we construct a proper Deligne-Mumford stack which we will show to be isomorphic to a choice of underlying algebraic stack obtained through the Maulik-Ranganthan construction. We use the word underlying here because what is constructed in \cite{MR} is a logarithmic algebraic stack and we impose no logarithmic structure on our space.

\subsection{Stacks and stability conditions}\label{Motivation}
Before we describe the expanded degenerations as stacks, we comment on the role of stacks in this problem and on the stability conditions defined in Section \ref{stab_conditions}.

\subsubsection*{Base change.}
Our aim is to construct degenerations of Hilbert schemes of points as good moduli spaces. In the proofs of Propositions \ref{universally closed} and \ref{separated 1}, we will use the valuative criterion to prove universal closure and separatedness. We will see that this argument holds only up to base change, which is why it is necessary for us to work with stacks instead of schemes.

\subsubsection*{Non-separated GIT quotient stacks.}
Although the GIT quotients $I^m_{[n],\cM}$ are projective and thus proper over $\AA^1$, their corresponding stack quotients are not necessarily proper. Indeed, the GIT quotient does not see the orbits of the group action themselves but the closures of these orbits. For example in Figure \ref{15}, the red pair of points and the blue pair of points are in the same orbit closure, so the GIT quotient considers them as equivalent, while the corresponding stack quotient regards them as belonging to separate orbits. This means that, in the stack, allowing for both pairs will break separatedness.

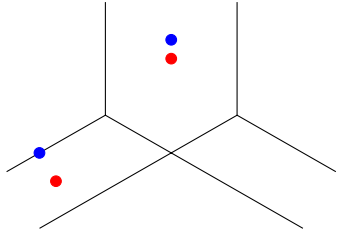
\begin{figure}
    \centering
    \begin{tikzpicture}[scale=1]
\draw   (-0.866,0.5) -- (0,0)
        (0,0) -- (0.866,0.5)
        
        (-0.866,0.5) -- (-0.866,2)
        (0.866,0.5) -- (0.866,2)
        (-1.732, -1) -- (0,0)
        (-2.165,-0.249) -- (-0.866,0.5)
        
        (1.732, -1) -- (0,0)
        
        (2.165,-0.249) -- (0.866,0.5);
        
        \filldraw[red] (-1.516,-0.375) circle (2pt) ;
        \filldraw[red] (0,1.25) circle (2pt) ;
        \filldraw[blue] (0,1.5) circle (2pt) ;
        \filldraw[blue] (-1.733,0) circle (2pt) ;
        
\end{tikzpicture}
    \caption{Non-separatedness in GIT stable locus.}
    \label{15}
\end{figure}

In the following sections, when studying quotient stacks, we will therefore want to consider the sublocus of the GIT stable stable locus containing only length $m$ zero-dimensional subschemes which are smoothly supported in a given fibre of $X[n]$. Building a compactification in which limits are represented by smoothly supported subschemes will also be useful for future applications as it allows us to break down the problem of a Hilbert scheme of $m$ points on a singular surface into products of Hilbert schemes of fewer than $m$ points on smooth components.

\subsubsection*{Patching together GIT stability conditions.} No single GIT quotient $I^m_{[n],\cM}$ contains all desired limits as smoothly supported subschemes. Therefore in the stack construction, the stability condition we define will draw on these local quotients, but globally will not correspond to one single GIT stability condition. We now define a notion that we will use in the following sections.

\begin{definition}\label{base codimension}
    We say that a fibre in some expanded degeneration $X[n]$ has \textit{base codimension} $k$ if exactly $k$ basis directions vanish at this fibre. This is independent of the value $n$.
\end{definition}

\subsubsection*{Making the expanded degenerations large enough.}
Finally, if we construct a unique GIT quotient in which not all limit subschemes are smoothly supported then the limits given by orbit closures containing only subschemes with singular support will not lie in a fibre of the expected base codimension. This gives an intuition that the degeneration we have chosen is too small. That being said, it can be useful to think about this GIT quotient if what we are trying to do is simply to resolve singularities in a way that preserves some good properties of the space, e.g.\ in the context of constructing minimal models for type III degenerations of Hilbert schemes of points on K3 surfaces.

\subsection{Expanded construction for stacks}

In this section we construct a stack of expansions $\mathfrak{C}$ and family over it $\mathfrak{X}$, keeping our notation as close as possible to that of \cite{LW}.

\subsubsection*{The stack $\mathfrak{C}$.}
In the following we define the stack $\mathfrak{C}$ identically to the stack of expanded degenerations defined by Li and Wu. For convenience, we recall the details of this construction here.

Let us consider $\AA^{n+1}$ with its natural torus action $\GG_m^{n}$ as defined above. We then impose some additional relations given by a collection of isomorphisms which we describe in the following. As before, we label elements of the base as $(t_1,\dots, t_{n+1})$. We start by defining the set
\begin{align*}
    [n+1] &\coloneqq \{1,\dots, n+1\}.
\end{align*}
Let $I\subseteq [n+1]$ and $I^\circ = [n+1]-I$ be the complement of $I$. For $|I|= r+1$, let
\[
\ind_{I}\colon [r+1] \longrightarrow  I \subset [n+1]
\]
and
\[
\ind_{I^\circ}\colon [n-r] \longrightarrow  I^\circ \subset [n+1]
\]
be the order preserving isomorphisms. Let
\[
\AA^{n+1}_{I} = \{ (t)\in \AA^{n+1} |\ t_i = 0,\ t_i\in I \} \subset \AA^{n+1}
\]
and
\[
\AA^{n+1}_{U(I)} = \{ (t)\in \AA^{n+1} |\ t_i \neq 0,\ t_i\in I^\circ \} \subset \AA^{n+1}.
\]
Then we have the isomorphism
\[
\Tilde{\tau}_I \colon (\AA^{r+1}\times \GG_m^{n-r}) \longrightarrow \AA^{n+1}_{U(I)}
\]
given by
\[
(a_1, \ldots, a_{r+1},\sigma_1,\dots, \sigma_{n-r}) \longmapsto
(t_1,\dots, t_{n+1}),
\]
where 
\begin{align*}
    & t_k = a_l, \ \mathrm{ if } \ \ind_{I}(l)=k, \\
    & t_k = \sigma_l, \ \mathrm{ if } \ \ind_{I^\circ}(l)=k.
\end{align*}
Then, given $I,I'\subset [n+1]$ such that $|I|=|I'|$, we define an isomorphism
\[
\Tilde{\tau}_{I,I'} = \Tilde{\tau}_{I} \circ \Tilde{\tau}_{I'}^{-1} \colon \AA^{n+1}_{U(I')} \longrightarrow \AA^{n+1}_{U(I)}.
\]
Recall from Section \ref{BUs} that we had natural inclusions \eqref{standard embedding1} 
\begin{align*}
    C[n] \longhookrightarrow C[n+1],
\end{align*}
which called \textit{standard embeddings} in \cite{LW}.

\medskip
Finally, we define $\mathfrak{U}^{n}$ to be the quotient $[\AA^{n+1}/\!\!\sim]$ by the equivalences generated by the $\GG_m^{n}$-action and the equivalences $\Tilde{\tau}_{I,I'}$ for pairs $I,I'$ with $|I|=|I'|$. We can define open immersions
\begin{align*}
    \mathfrak{U}^{n} \longrightarrow \mathfrak{U}^{n+1},
\end{align*}
induced by the standard embeddings. Let $\mathfrak{U} \coloneqq \lim\limits_{\to} \mathfrak{U}^{n}$ be the direct limit over $n$ and let $\mathfrak{C} \coloneqq C\times_{\AA^1} \mathfrak{U} $.

\subsubsection*{The stack $\mathfrak{X}$.}

Let $X[n]\to C[n]$ be as in Section \ref{Second_constr} and recall that $\pi\colon X[n] \to X$ is the projection to the original family. Let 
\[
\Bar{\tau}_I \colon C[m] \longhookrightarrow C[n]
\]
be the standard embedding. Then the induced family
\(
(\Bar{\tau}_I^* X[n],\Bar{\tau}_I^* \pi)
\)
is isomorphic to $(X[m],\pi)$ over $C[m]$. The equivalences on $\mathfrak{U}^{n}$ lift to $C$-isomorphisms of fibres.

\medskip
We define $\mathfrak{X}^{n}$ to be the quotient $[X[n]/\!\!\sim]$ by the equivalences generated by the $\GG_m^{n}$-action and equivalences lifted from $\mathfrak{U}^{n}$. There are natural immersions of stacks
\begin{align*}
    \mathfrak{X}^{n} \longrightarrow \mathfrak{X}^{n+1},
\end{align*}
induced by the immersions $\mathfrak{U}^{n} \to \mathfrak{U}^{n+1}$. Finally, we define $\mathfrak{X} = \lim\limits_{\to} \mathfrak{X}^{n}$ to be the direct limit over $n$. It is an Artin stack.

\subsection{Stability conditions.}\label{stab_conditions}

\subsubsection*{Restricting to the smoothly supported locus.}
We start by examining some stability conditions on the scheme $H^m_{[n]}$. We have defined several $G$-linearised line bundles $\cM$ on this space. As before, let us denote by $H^{m,ss}_{[n],\cM}$ and $H^{m,s}_{[n],\cM}$ the corresponding GIT semistable and stable loci respectively. As discussed in Section \ref{Motivation}, considering the GIT stable locus does not give us a separated quotient stack, among other reasons because it contains some subschemes which are not smoothly supported.

\begin{proposition}
    The restrictions of the loci $H^{m,ss}_{[n],\cM}$ and $H^{m,s}_{[n],\cM}$ to the loci of smoothly supported subschemes are $G$-invariant open subschemes.
\end{proposition}

\begin{proof}
    Recall that the relative Hilbert scheme of $m$ points on $X[n] \to C[n]$, which we denoted $H^m_{[n]}$, is the scheme which represents the functor 
    \[
    h\colon \ksch^{\opp} \longrightarrow \Sets,
    \]
    where $\ksch^{\opp}$ is the category of $k$-schemes. This functor associates to any $k$-scheme $B$ the set of flat families over $B$ of subschemes of fibres of $X[n]$ over $C[n]$. Restricting $X[n]$ to the smooth locus of its fibres yields a family of open subschemes $X[n]^{\sm}$ over $C[n]$, and we can similarly define a Hilbert functor $h_{\sm}$ on this family. There is a morphism from the corresponding Hilbert scheme $H^m_{[n],\sm}\coloneqq \Hilb (X[n]^{\sm}/C[n])$ to $H^m_{[n]}$ which is clearly a monomorphism and it is étale since deformations of smoothly supported subschemes are smoothly supported. We could also note that the complement of $H^m_{[n],\sm}$ in $H^m_{[n]}$ is closed by the valuative criterion since the limit of any subscheme with part of its support in the singular locus of a fibre must also have part of its support in the singular locus of a fibre.
    
    We remark that since the smooth locus of the fibres of $X[n]$ is $G$-invariant, restricting the functor to this locus preserves the $G$-invariance. The restrictions of the semistable and stable loci to the loci of smoothly supported subschemes therefore yield $G$-invariant open subschemes.
\end{proof}

\emph{Notation.} We denote by $H^{m, s}_{[n],\cM,\sm} $ and $ H^{m, ss}_{[n],\cM,\sm}$ the loci of GIT stable and semistable subschemes which are smoothly supported.

\medskip
We have the following inclusions:
\[
H^{m, s}_{[n],\cM,\sm} \subset H^{m, ss}_{[n],\cM,\sm} \subset H^{m, ss}_{[n],\cM}.
\]

\subsubsection*{Li-Wu stability.}
We recall here the notion of stability used in \cite{LW}, in order to compare it with the GIT stability and construct an appropriate stability condition for this case.

\begin{definition}
    Let $X[n]_0$ be a fibre of $X[n]$ over a closed point and let $D$ denote the singular locus of $X[n]_0$. A subscheme $Z$ in $X[n]_0$ is said to be \emph{admissible} if the morphism 
    \[
    \cI_Z \otimes \cO_D \to \cO_D
    \]
    is injective, where $\cI_Z$ is the ideal sheaf of $Z$, i.e.\ when $Z$ is normal to $D$. A family $Z$ in $X[n]$ is \emph{admissible} if it is admissible in every fibre over a closed point. A subscheme $Z$ in $\mathfrak{X}$ is \emph{Li-Wu stable} (LW stable) if and only if it is admissible and has finite automorphism group.
\end{definition}

For a length $m$ zero-dimensional scheme $Z$ in $\mathfrak{X}$, the admissible condition will mean that none of the points in the support of $Z$ lie in the singular locus of the given fibre. The finite automorphism condition means that the stabiliser of $Z$ with respect to the torus action we defined on the blow-ups must be finite. Denote the LW stable locus by $H^m_{[n],\LW}$. Note that we have an inclusion
\begin{equation}\label{GIT stable in LW stable locus}
    H^{m, s}_{[n],\cM,\sm}\subset H^m_{[n],\LW}
\end{equation}
for all $G$-linearised line bundles $\cM$ on $H^m_{[n]}$ since, if points are GIT stable, they must have finite stabilisers. This inclusion no longer holds for the GIT semistable locus. This is a strict inclusion as the LW stability is clearly a weaker condition than the GIT strict stability with smooth support.

\subsubsection*{Modified GIT stability.} As stated above, we only want to allow length $m$ zero-dimensional subschemes to be stable if their support lies in the smooth locus of a fibre. However, restricting the GIT stability condition to this locus makes the space of stable subschemes no longer uiversally closed. Indeed, there is no single GIT condition which can represent all desired length $m$ zero-dimensional subschemes as smoothly supported subschemes. We must therefore define a modified GIT stability condition which patches together several GIT stability conditions in order to obtain the desired stable locus.

\begin{definition}
    Let $[Z]\in H^m_{[n]}$. We say that $Z$ is \textit{weakly strictly stable} (WS stable) if there exists a $G$-linearised ample line bundle on $H^m_{[n]}$ with respect to which $Z$ is stable. We denote the WS stable locus in $H^m_{[n]}$ by $H^m_{[n],\WS}$. We shall denote by $H^m_{[n],\SWS}$ the locus of WS stable smoothly supported subschemes.
\end{definition}

We may write $H^m_{[n],\SWS}$ as the union
\[H^m_{[n],\SWS} \coloneqq \bigcup_{\cM} H^{m, s}_{[n],\cM,\sm}\]
over all choices of $G$-linearised line bundle $\cM$. It is then clear that we have an inclusion $H^m_{[n],\SWS} \subset H^m_{[n],\LW}$ by the inclusion \eqref{GIT stable in LW stable locus}. We will now want to compare these stability conditions on the stack $\mathfrak{X}$, so we will need to extend our definition of WS stability to this stack.

\medskip
Given a $C$-scheme $S$, an object of $\mathfrak{X}(S)$ is a pullback family $\xi^* X[n]$ for a morphism
\[
\xi \colon S \to C[n].
\]
Now we describe WS stability on the stack $\mathfrak{X}$.

\begin{definition}
    A pair $(Z,\mathcal{X})$, where $Z$ is a family of length $m$ zero-dimensional subschemes in $\cX \in \mathfrak{X}(S)$, is said to be \emph{WS stable} if and only if $\cX \coloneqq \xi^* X[n]$ for some morphism $\xi \colon S \to C[n]$ and there exists some $G$-linearised ample line bundle on $H^m_{[n]}$ which makes $Z$ be GIT stable. We will say that $Z$ is SWS stable if it is smoothly supported and is WS stable.
\end{definition}

\begin{remark}
    Note that we are slightly abusing notation in the above definition, by asking for $Z$ to be GIT stable in $H^m_{[n]}$, when $Z$ is defined in $\cX$, and it is in fact $\xi_* Z$ which must be GIT stable in $H^m_{[n]}$. This is a harmless simplification as it will always be clear from context what we mean. We continue to use it throughout the work for convenience, especially where the map $\xi$ has not been specified.
\end{remark}

\subsubsection*{Stacks of stable objects.}
Let us denote by $\mathfrak{M}^m_{\SWS}$ and $\mathfrak{M}^m_{\LW}$ the stacks of SWS and LW stable length $m$ zero-dimensional subschemes in $\mathfrak{X}$ respectively. Let $S$ be a $C$-scheme. An object of $\mathfrak{M}^m_{\SWS}(S)$ is defined to be a pair $(Z,\cX)$, where $\cX\in \mathfrak{X}(S)$ and $Z$ is an $S$-flat SWS stable family in $\cX$. Similarly, an object of $\mathfrak{M}^n_{\LW}(S)$ is a pair $(Z,\cX)$, where $\cX\in \mathfrak{X}(S)$ and $Z$ is an $S$-flat LW stable family in $\cX$.

\begin{remark}
    Note that it does not make sense in general to speak of Maulik-Ranganathan stability (MR stability) without defining an appropriate notion of tube components on our stacks as in Section \ref{MR_section}. In this specific setting, however, we will see that there is no need to specify tube components as the stacks $\mathfrak{M}^m_{\SWS}$ and $\mathfrak{M}^m_{\LW}$ are already proper. The LW stability which we extended to our situation will therefore be equivalent to MR stability on $\mathfrak{X}$. In this setting we may therefore use both terminologies interchangeably.
\end{remark}

As is briefly discussed in Section \ref{further results}, we can also make constructions, equivalent to some of the constructions of Maulik and Ranganathan, which require choices of representatives of limit subschemes and labelling of components as tube components. As the construction we make here requires the minimal amount of choice (the only choice was in choosing to blow up $Y_1$ and $Y_2$ but not $Y_3$ at the very start) we shall refer to it as the canonical construction.

\section{The canonical moduli stack}\label{canonical moduli stack}

In this section we show that the stacks $\mathfrak{M}^m_{\SWS}$ and $\mathfrak{M}^m_{\LW}$ are proper and Deligne-Mumford and that they are in fact isomorphic.

\subsection{Properness and Deligne-Mumford property}

In this section, we show that the stacks $\mathfrak{M}^m_{\SWS}$ and $\mathfrak{M}^m_{\LW}$ are universally closed, separated and have finite automorphisms. Before we give these proofs, we make the following definition.

\begin{definition}\label{extension}
    Let $S\coloneqq \Spec R \to C$, where $R$ is some discrete valuation ring and let $\eta$ denote the generic point of $S$. Now, let $(Z,\cX)$ be a pair where $\cX\in \mathfrak{X}(S)$ and $Z$ is an $S$-flat family of length $m$ zero-dimensional subschemes in $\cX$. Let $S'\to S$ be some finite base change and denote the generic and closed points of $S'$ by $\eta'$ and $\eta'_0$ respectively. We say that a pair $(Z'_{\eta_0'},\cX'_{\eta_0'})$ is an \textit{extension} of $(Z_\eta,\cX_\eta)$ if there exists such a base change and $(Z'_{\eta_0'},\cX'_{\eta_0'})$ is the restriction to $\eta'_0$ of some $S$-flat family $(Z',\cX')$ with $\cX'\in\mathfrak{X}(S')$ such that $Z_\eta \times_{\eta} \eta' \cong Z'_{\eta'}$ and $\cX_\eta \times_{\eta} \eta' \cong \cX'_{\eta'}$.
\end{definition}

\begin{proposition}\label{universally closed}
The stack $\mathfrak{M}^m_{\SWS}$ is universally closed.
\end{proposition}

\begin{proof}
Let $S\coloneqq \Spec R \to C$, where $R$ is some discrete valuation ring with uniformising parameter $w$ and quotient field $k$. We denote by $\eta$ and $\eta_0$ the generic and closed points of $S$ respectively. Let $(Z, \cX)$ be an $S$-flat family of length $m$ zero-dimensional subschemes such that $\cX\coloneqq \xi^* X[r]\in \mathfrak{X}(S)$ for some morphism $\xi\colon S\to C[r]$ and $(Z_\eta, \cX_\eta)\in \mathfrak{M}^m_{\SWS}(\eta)$. Additionally, we assume that all basis directions are invertible at $\xi(\eta)$, i.e.\ $\cX_\eta$ has base codimension zero. As mentioned at the end of this proof, the other case is treated in Proposition \ref{limits of special elements}. We show that there exists a finite base change $S'\coloneqq \Spec R'\to S$, for some discrete valuation ring $R'$ and a pair $(Z', \cX')\in \mathfrak{M}^m_{\SWS}(S')$ satisfying the following condition. We denote by $\eta'$ and $\eta_0'$ the generic and closed points of $S'$ respectively. Then $S'$ and $(Z', \cX')$ are chosen such that we have an equivalence $\cX'_{\eta'} \cong \cX_\eta \times_\eta \eta'$ which induces an equivalence $Z'_{\eta'} \cong Z_\eta \times_\eta \eta'$.

\medskip
Let $\mathfrak{X}(S)$ be defined by the equation $xyz= cw^h$, where $w$ is the uniformising parameter of $R$ as above and $c$ is a unit in $R$. The subscheme $Z$ is a union of irreducible components $Z_i$ whose defining equations we will want to express in terms of the uniformising parameter. We therefore start by taking an appropriate base change $S'\coloneqq \Spec R'\to S$, which maps $u^k \to w^h$, where $u$ is the uniformising parameter of $R'$ and where $u$ is chosen such that each $Z_i$ can be written locally in terms of its $x,y$ and $z$ coordinates as
\begin{equation}
    \{ (c_{i,1}u^{e_{i,1}}, c_{i,2}u^{e_{i,2}}, c_{i,3}u^{e_{i,3}})\},
\end{equation}
for some $e_{i,j}\in \ZZ$ and some units $c_{i,j}$ in $R'$. Note that $\mathfrak{X}(S')$ is defined by the equation $xyz = cu^k$ and we therefore have the equality
\[
c_{i,1}c_{i,2}c_{i,3} u^{e_{i,1}+ e_{i,2} + e_{i,3}} = cu^k
\]
for all $i$.

\medskip
\emph{Tropical perspective.} The choice of uniformising parameter $w$ corresponds to a choice of height of the triangle $\trop(X_0)$ within $\trop(X)$. We may then examine the rays in $\trop(X)$ corresponding to the image $\trop(Z_\eta)$ of $Z$ under the tropicalisation map (see Section \ref{MR_section}) for definitions. If the vertices given by $\trop(Z_\eta)\cap \trop(X_0)$ do not already lie on integral points of the cone, then we must change the height of the triangle within $\trop(X)$ until the intersection vertices $\trop(Z_\eta)\cap \trop(X_0)$ become integral. This dictates exactly what base change $S'\to S$ to make as the edge length of the new triangle is given by $e_{i,1}+ e_{i,2} + e_{i,3}$. The integral intersection vertices on this triangle corresponding to each $Z_i$ will be given by $(e_{i,1}, e_{i,2}, e_{i,3})$.

\medskip
We now form an ordered list $(d_1u^{e_1}, \dots, d_{2m}u^{e_{2m}})$, where we arrange all values $c_{i,1}u^{e_{i,1}}$ and $(c_{i,2})^{-1}c u^{k-e_{i,2}}$ from smallest to largest power of $u$. If two terms have the same power of $u$, we may place them in any order. We shall now inductively construct an element $(t_1,\dots, t_{n+1})$ of $\AA^{n+1}$ determining a morphism $\xi\colon S'\to C[n]$ such that the pullback $\xi^*X[n]$ defines the family $\cX'$. We start by setting
\[
(t_1,t_2) = (d_1u^{e_1}, (d_1u^{e_1})^{-1}cu^k).
\]
If $e_1 = e_2$, then we do not include $d_2u^{e_2}$ and move on to $e_3$. If $e_1 \neq e_2$, however, we set
\[
(t_1,t_2,t_3) = (d_1u^{e_1}, (d_1u^{e_1})^{-1} d_2u^{e_2}, (d_2u^{e_2})^{-1} cu^k).
\]
We continue to iterate this process in the following way. Assume we have $(t_1,\dots, t_j)$, where $t_j = (d_lu^{e_l})^{-1} cu^k$. Then, if $e_{l+1}\neq e_l$, we write
\[
(t_1,\dots,t_j,t_{j+1}) = (d_1u^{e_1}, \dots, (d_lu^{e_l})^{-1} d_{l+1}u^{e_{l+1}}, (d_{l+1}u^{e_{l+1}})^{-1} cu^k),
\]
 and if $e_{l+1}= e_l$, then we move on to $l+2$ without including $d_{l+1}u^{e_{l+1}}$ in the expression. We iterate this until we find
 \begin{equation}\label{base expression}
     (t_1,\dots, t_{n+1}) = (f_1u^{g_1},\dots, f_{n+1}u^{g_{n+1}})
 \end{equation}
 which has exactly one entry for each different power of $u$ in the list $(d_1u^{e_1}, \dots, d_{2m}u^{e_{2m}})$.

 \medskip
 We now denote by $\pi_n\colon C[n] \to \AA^{n+1}$ the natural projection. The morphism $\xi\colon S'\to C[n]$ is defined by the condition that
 \begin{equation}\label{flat family}
     \pi_n\circ \xi = (f_1u^{g_1},\dots, f_{n+1}u^{g_{n+1}}).
 \end{equation}
 We may then define $\cX' \coloneqq \xi^* X[n]$ and let $Z' \coloneqq Z\times_{\cX} \cX'$. We show now that this satisfies all the necessary conditions.

 \medskip
 Since $\cX \in \mathfrak{X}(S)$ is a pullback $\cX = \xi^* X[r]$ for some $r$, where $\xi \colon S \to C[r]$ is given by a similar expression to \eqref{flat family}, then we have that $\cX'_{\eta'} \cong \cX_\eta \times_\eta \eta'$. Indeed, over the generic point, the uniformising parameter is invertible and any two expressions $(t_1,\dots,t_{l})$ and $(t_1,\dots,t_{l'})$ are equivalent in $\mathfrak{C}$ up to the equivalences of this stack if they have the same product $t_1\cdots t_l = t_1\cdots t_{l'}$. But in our case this product was chosen to be identical up to the base change factor.

\medskip
 Moreover, the expression \eqref{flat family} is chosen precisely to give an $S'$-flat extension of $Z\times_{\eta}\eta'$ where all points of the support of this extension lie in the smooth locus of the fibre $\cX'_{\eta_0's}$. Finally, the expression \eqref{flat family} ensures that we have expanded out the $\Delta$-components in the fibre $\cX'_{\eta_0}$ in such a way that every expanded $\Delta$-component in this fibre contains some point of the support of $Z'$. By Theorem \ref{stability theorem}, such a configuration will be stable with respect to some GIT stability condition on $H^m_{[n]}$.

 \medskip
 The above discussion shows that if $(Z_\eta, \cX_\eta)$ is pulled back from a fibre above a point $(t_1,\dots, t_{n+1})$ in some $C[n]$ whose entries are all invertible, then $(Z_\eta, \cX_\eta)$ has an SWS stable extension. See Proposition \ref{limits of special elements} for a proof that there exists an extension if $\cX_\eta$ is a modified special fibre, i.e.\ if some of the entries of $(t_1,\dots, t_{n+1})$ are not invertible.
\end{proof}

\begin{corollary}
    The stack $\mathfrak{M}^m_{\LW}$ is universally closed.
\end{corollary}

\begin{proof}
    As every SWS stable subscheme must be LW stable, the existence of limits in $\mathfrak{M}^m_{\LW}$ follows from the existence of limits in $\mathfrak{M}^m_{\SWS}$.
\end{proof}

\begin{proposition}\label{separated 1}
    The stacks $\mathfrak{M}^m_{\SWS}$ and $\mathfrak{M}^m_{\LW}$ are separated.
\end{proposition}

\begin{proof}
    Let $S \coloneqq \Spec R\to C$, where $R$ is a discrete valuation ring with uniformising parameter $u$. Let $\eta$ denote the generic point of $S$ and $\eta_0$ its closed point. Now, assume that there are two pairs $(Z,\cX)$ and $(Z',\cX')$ in $\mathfrak{M}^m_{\SWS}(S)$ such that $(Z_\eta,\cX_\eta) \cong (Z'_\eta,\cX'_\eta)$. We will show that it must follow that $(Z_{\eta_0},\cX_{\eta_0}) \cong (Z'_{\eta_0},\cX'_{\eta_0})$. Similarly to the proof of Proposition \ref{universally closed}, we assume that $\cX_\eta$ has base codimension zero. The other case is treated in Proposition \ref{limits of special elements}.

    \medskip
    We may assume that $S$ is chosen so that the $i$-th irreducible component of $Z$ is given in terms of its local coordinates $x,y$ and $z$ by
    \begin{equation}
        \{ (c_{i,1}u^{e_{i,1}}, c_{i,2}u^{e_{i,2}}, c_{i,3}u^{e_{i,3}})\},
    \end{equation}
    and the $i$-th irreducible component of $Z'$ is given in terms of its local coordinates $x$, $y$ and $z$ by
    \begin{equation}
        \{ (d_{i,1}u^{f_{i,1}}, d_{i,2}u^{f_{i,2}}, d_{i,3}u^{f_{i,3}})\}.
    \end{equation}
    Since the equivalences of the stack fix $x$, $y$ and $z$ and we know that $(Z_\eta,\cX_\eta) \cong (Z'_\eta,\cX'_\eta)$, it must therefore follow that $Z$ and $Z'$ have the same number of irreducible components. Moreover, if these components are labelled in a compatible way, then $c_{i,1} = d_{i,1}$ and $e_{i,1} = f_{i,1}$ for all $i$. But now, by flatness, each $Z_i$ and $Z_i'$ component must satisfy the equations
    \begin{align}
        x &= c_{i,1}u^{e_{i,1}}, \label{x equation} \\
        y &= c_{i,2}u^{e_{i,2}}, \label{y equation} \\
        z &= c_{i,3}u^{e_{i,3}}, \label{subscheme equations}
    \end{align}
    also above the closed point. If more than one element of the set $\{ e_{i,1}, e_{i,2}, e_{i,3} \}$ is nonzero, then this implies that either $Z_i$ and $Z_i'$ are not smoothly supported or are supported in a component blown-up along the vanishings of both sides of the above components. The stability condition forces $Z_i$ and $Z_i'$ to be smoothly supported, so the latter must be true. Moreover, since in our construction we have chosen to do our blow-ups along the vanishing of $x$ and the vanishing of $y$, this implies that $Z_i$ and $Z_i'$ must be supported in a component blown up along the ideals $\langle x, cu^{e_{i,1}} \rangle$ and $\langle y, c'u^{e_{i,2}} \rangle$ over the closed point $\eta_0$, for some units $c$ and $c'$ in $R$.

    \medskip
    Note that different values of $c$ and $c'$ will cause the relevant points of the support of $Z_i$ and $Z_i'$ to take on different values in the interior of the $\PP^1$ introduced by each blow-up. Since the $\GG_m$-action imposed on the $\PP^1$ identifies all points within the interior of a $\PP^1$, this choice makes no difference.
    
    \medskip
    Notice also that blowing up along $\langle x, cu^{e_{i,1}} \rangle$ and blowing up along $\langle yz, (cu^{e_{i,1}})^{-1} du^k \rangle$, where $\mathfrak{X}(S)$ is defined by the equation $xyz = du^k$, are the same. This allows us to obtain the equation \eqref{subscheme equations}.

    \medskip
    We have established that both $Z_i$ and $Z_i'$ must be supported in the blown-up components described above for all $i$ such that more than one element of the set $\{ e_{i,1}, e_{i,2}, e_{i,3} \}$ is nonzero. We also know that by the stability conditions the pairs $(Z_{\eta_0},\cX_{\eta_0})$ and $(Z'_{\eta_0},\cX'_{\eta_0})$ cannot have an expanded component containing no point of the support. Let $\pi_n\colon C[n] \to \AA^{n+1}$ denote the natural projection, as above. It follows that the morphism
    \begin{equation}\label{unique flat family}
        \pi_n\circ \xi = (h_1u^{g_1},\dots, h_nu^{g_n})\colon S \to C[n] \to \AA^{n+1}
    \end{equation}
    defining the family $\cX = \xi^* X[n]$ is uniquely determined up to the choices of units $h_i$ in $R$ and embeddings by the standard embeddings. If the family $\cX'$ is defined by a morphism as in \eqref{unique flat family} but with different nonzero $g_i$ values, then either $Z'_{\eta_0}$ is not smoothly supported in $\cX'_{\eta_0}$ or $\cX'_{\eta_0}$ has an expanded component containing no point of the support of $Z'_{\eta_0}$. This shows uniqueness of limits.
\end{proof}

\subsubsection*{Existence and uniqueness of limits for special objects.}
We need to establish some definitions before we prove the following auxiliary result on existence and uniqueness of limits for special elements, i.e. when the fibre $\cX_\eta$ over the generic point of $S$ is a modified special fibre itself.

\medskip
Let $S\coloneqq \Spec R$ for some discrete valuation ring $R$, let $\eta$ be its generic point and take $(Z_\eta, \cX_\eta) \in \mathfrak{N}^m_{\SWS}(\eta)$ (or $\mathfrak{N}^m_{\LW}(\eta)$). Here $\eta$ is not necessarily pulled back from a point in $C[n]$ with only invertible basis directions (i.e.\ $\cX_\eta$ may be a modified special fibre). We can consider the image $\trop(Z_\eta)$ under the tropicalisation map given in Section \ref{MR_section} as a collection of rays in $\trop(X)$.

\medskip
Again, here we are abusing notation slightly: the object we are considering is actually the image of $\xi_*(Z_\eta)$ under the tropicalisation map, where $\xi\colon S \to C[n]$. Similarly, we will write $\trop(\cX_\eta)$ to mean the tropicalisation of the pushforward along $\xi$.

\medskip
In order to construct an extension $(Z'_{\eta'_0}, \cX'_{\eta'_0})$ of $(Z_\eta, \cX_\eta)$ such that $Z'_{\eta'_0}$ is smoothly supported in $\cX'_{\eta'_0}$, each ray making up $\trop(Z_\eta)$ (or, equivalently, the corresponding vertex in $\trop(X_0)$) must correspond to a nonempty bubble in $\cX'_{\eta'_0}$. This effectively determines all nonempty bubbles which must exist in $(Z'_{\eta'_0}, \cX'_{\eta'_0})$, but in order for these rays to appear as part of a polyhedral subdivision of $\trop(X)$, we might need to add more rays (or vertices in $\trop(X_0)$) corresponding to the empty bubbles in the pair $(Z'_{\eta'_0}, \cX'_{\eta'_0})$.

\begin{definition}
    In the notation of the above paragraph, we will call $(Z'_{\eta'_0}, \cX'_{\eta'_0})$ an \textit{associated pair} for a collection of rays in $\trop(X)$ (or vertices in $\trop(X_0)$) if these rays (or vertices) correspond exactly to the non-empty bubbles in $(Z'_{\eta'_0}, \cX'_{\eta'_0})$ in the manner described above.
\end{definition}

For $I\subset [n]$, we denote by $X[n]_I$ the fibres where $t_i$ vanish for all $i\in I$. Now we define the necessary condition for compatibility of limits in the stacks of stable objects.

\begin{definition}
    Let $(Z_\eta,\cX_\eta)\in \mathfrak{M}^m_{\SWS}(\eta)$ (or $\mathfrak{M}^m_{\LW}(\eta)$) be any pair over the generic point of some $S\coloneqq \Spec R$, for some discrete valuation ring $R$ as before. Moreover, let $\cX_\eta$ be the generic fibre of $\cX\coloneqq \xi^* X[n]_I$ for some nonempty set $I$, i.e.\ $\cX_\eta$ is pulled back from some modified special fibre. If, for any associated pair $(Z'_{\eta'_0}, \cX'_{\eta'_0})$ of $\trop(Z_\eta)$ in $\mathfrak{M}^m_{\SWS}$ (or $\mathfrak{M}^m_{\LW}$), the tropicalisation $\trop(\cX'_{\eta'_0})$ is a subdivision of $\trop(\cX_\eta)$, then we say that $\mathfrak{M}^m_{\SWS}$ (or $\mathfrak{M}^m_{\LW}$) is \emph{tropically compatible}.
\end{definition}

\begin{lemma}\label{exists unique associated pair}
    Let $(Z_\eta,\cX_\eta)\in \mathfrak{M}^m_{\SWS}(\eta)$ (or $\mathfrak{M}^m_{\LW}(\eta)$) be as above. Then $\trop(Z_\eta)$ has a unique associated pair which is SWS (or LW) stable.
\end{lemma}

\begin{proof}
    This is clear from the construction of $X[n]$. Given any configuration of vertices in $\trop(X_0)$, we have allowed, by our restrictive choice of blow-ups in the construction of $\mathfrak{X}$, exactly one way of adding edges to the triangle $\trop(X_0)$ such that each of these vertices land on the intersection of at least two edges and such that the corresponding extension of $(Z_\eta,\cX_\eta)$ is stable.
\end{proof}

Let $S\coloneqq \Spec R\to C$, where $R$ is a discrete valuation ring and $\eta$ is the generic point of $S$. Let $(Z_\eta,\cX_\eta)\in \mathfrak{M}^m_{\SWS}(\eta)$ (or $\mathfrak{M}^m_{\LW}(\eta)$). We have shown in the proofs of Propositions \ref{separated 1} and \ref{universally closed} that if $\cX_\eta$ is pulled back from a fibre in some $X[n]$ over a point $(t_1,\dots,t_{n+1})$ whose entries are all nonzero, then $(Z_\eta,\cX_\eta)$ has a stable extension in $\mathfrak{M}^m_{\SWS}$ (or $\mathfrak{M}^m_{\LW}$). We now prove the following statement, to complete the proofs that $\mathfrak{M}^m_{\SWS}$ and $\mathfrak{M}^m_{\LW}$ are universally closed and separated.

\begin{proposition}\label{limits of special elements}
    In the notation of the above paragraph, let $(Z_\eta,\cX_\eta)\in \mathfrak{M}^m_{\SWS}(\eta)$ (or $\mathfrak{M}^m_{\LW}(\eta)$) and assume that $\cX_\eta$ is the pullback of $X[n]_I$ over the generic point along some morphism $\xi\colon S\to C[n]_I$ and some nonempty set $I$. The fibre $\cX_\eta$ is therefore a modified special fibre. Then there exists an SWS (or LW) stable extension of $(Z_\eta,\cX_\eta)$.
\end{proposition}

\begin{proof}
We split the proof into the following two cases. The first case is where a point $P$ of the support of $Z_\eta$ has one or more of its local coordinates $x,y$ or $z$ tending to zero. The second case is where a point $P$ of the support of $Z_\eta$ has fixed $x,y$ and $z$ values but one or more of its $(x_0^{(i)}:x_1^{(i)})$ or $(y_0^{(i)}:y_1^{(i)})$ coordinates tends towards $(1:0)$ or $(0:1)$.

\medskip
We start by proving existence and uniqueness of limits in the first case using the valuative criterion. Let $V$ denote the irreducible component of $\cX_\eta$ in the interior of which $P$ lies. Notice that since $P$ tends towards a codimension greater or equal to one stratum of $\cX$, then in order for its limit to be smoothly supported in an extension of $(Z_\eta,\cX_\eta)$, it will be necessary to expand out at least one $\Delta$-component in this extension. There exists a smoothing from the interior of $V$ in the fibre over the generic point to the interior of this expanded $\Delta$-component in such an extension of $(Z_\eta,\cX_\eta)$ if and only if this $\Delta$-component is equal to $V$ in the fibre over the generic point. Moreover, if there is no such $\Delta$-component equal to $V$, then none of the $x,y$ or $z$ coordinates can tend towards zero (because both sides of the defining equations must tend towards zero).

\medskip
By Lemma \ref{exists unique associated pair}, we know that there exists a unique SWS (or LW) stable associated pair $(Z'_{\eta'_0},\cX'_{\eta'_0})$ for the image $\trop(Z_\eta)$ of $Z_\eta$ in $\trop(X)$. Precisely, this means that there exists some base change $S'\to S$ and some SWS (or LW) stable pair $(Z'_{\eta'_0},\cX'_{\eta'_0})$ over the closed point $\eta'_0$ of $S'$ such that the non-empty bubbles of $\cX'_{\eta'_0}$ correspond exactly to the rays in $\trop(Z_\eta)$.

We must now show that the equivalences $\cX'_{\eta'} \cong \cX_\eta \times_\eta \eta'$ and $Z'_{\eta'} \cong Z_\eta \times_\eta \eta'$ hold.
But this follows from the fact that $\mathfrak{M}^m_{\SWS}$ and $\mathfrak{M}^m_{\LW}$ are tropically compatible by construction. Indeed, any modified special fibre in $\mathfrak{X}$ can be obtained from any modified fibre of lower base codimension by a sequence of blow-ups. Every vertex in $\trop(\cX_\eta)$ is therefore a vertex in $\trop(\cX'_{\eta'_0})$, i.e.\ $\cX'_{\eta'_0}$ can be seen as a blow-up of $\cX_\eta$. This tells us that $(Z'_{\eta'_0},\cX'_{\eta'_0})$ can be seen as the restriction to the closed point $\eta'_0$ of an $S'$-flat family $(Z',\cX')\in \mathfrak{M}^m_{\SWS}(S')$ (or $\mathfrak{M}^m_{\LW}(S')$) such that $\cX'_{\eta'} \cong \cX_\eta \times_\eta \eta'$ and $Z'_{\eta'} \cong Z_\eta \times_\eta \eta'$.

\medskip
Now, let us dicuss the second case. Denote again by $V$ the irreducible component of $\cX_\eta$ in the interior of which $P$ lies. Firstly, let us assume that any other point of the support of $Z_\eta$ lying in the interior of $V$ shares the same equations as $P$ up to multiplication by a constant. In particular, for all these points the same coordinates $(x_0^{(i)}:x_1^{(i)})$ or $(y_0^{(i)}:y_1^{(i)})$ will be fixed and the same will vary (recall that in this case we are assuming already that all $x,y$ and $z$ coordinates are fixed). By flatness, it therefore follows that all points of the support of $Z_\eta$ which lie in the interior of $V$ will tend to the interior of the same irreducible component in any extension of $(Z_\eta,\cX_\eta)$. Any extension in which an additional $\Delta$-component is expanded out, i.e.\ in which an additional basis direction is set to zero, cannot be stable, since it would necessarily have an empty expanded component which would destabilise the pair.

\medskip
Notice that any $(x_0^{(i)}:x_1^{(i)})$ or $(y_0^{(i)}:y_1^{(i)})$ which are fixed for one of these points contained in the interior of $V\subset\cX_\eta$ must be fixed for all of them. We may therefore choose a representative $(Z'_\eta,\cX'_\eta)$ in the same equivalence class as $(Z_\eta,\cX_\eta)$ such that the coordinates $x_0^{(i)}/x_1^{(i)}$, $x_1^{(i)}/x_0^{(i)}$, $y_0^{(i)}/y_1^{(i)}$ or $y_1^{(i)}/y_0^{(i)}$ which are allowed to vary in $Z$ are invertible only outside of $V\subset\cX'_\eta$. Then the limit of $(Z'_\eta,\cX'_\eta)$ is $(Z'_\eta,\cX'_\eta)$ itself. And by the above this equivalence class gives us the only stable limit for such a family.

\medskip
Now, if two points $P_0$ and $P_1$ of the support of $Z_\eta$ which lie in $V\subset \cX_\eta$ have different defining equations up to multiplication by a constant, it means that there are some coordinates $(x_0^{(i)}:x_1^{(i)})$ or $(y_0^{(i)}:y_1^{(i)})$ which are fixed for $P_0$ but not $P_1$ and vice versa. This means that, as one or the other of these coordinates tends toward $(1:0)$ or $(0:1)$, the points $P_0$ and $P_1$ must tend towards different irreducible components. Similarly to the first case, by Lemma \ref{exists unique associated pair}, we know that there exists a unique associated pair $(Z'_{\eta'_0},\cX'_{\eta'_0})$ for $\trop(Z_\eta)$ and by the tropical compatibility condition shown above, we know that every vertex in $\trop(\cX_\eta)$ is a vertex in $\trop(\cX'_{\eta'_0})$. As before this shows that $(Z'_{\eta'_0},\cX'_{\eta'_0})$ is the unique SWS (or LW) stable extension of $(Z_\eta,\cX_\eta)$.

\medskip
If $V\subset \cX_\eta$ contains more than two points of the support which have different equations up to multiplication by a constant, we can just repeat the steps of the previous paragraph until we find a stable extension. It will be unique as it will be the unique pair associated to $\trop(Z_\eta)$.
\end{proof}

\subsubsection*{Deligne-Mumford property.} Finally we show that both stacks of stable objects constructed have finite automorphisms.
\begin{proposition}\label{finite autom}
    The stacks $\mathfrak{M}^m_{\LW}$ and $\mathfrak{M}^m_{\SWS}$ have finite automorphisms.
\end{proposition}

\begin{proof}
    On the stack $\mathfrak{M}^m_{\LW}$ this is immediate from the definition of LW stability. Since the SWS stable locus is a subset of the LW stable locus, it follows that $\mathfrak{M}^m_{\SWS}$ must also have finite automorphisms. Alternatively, one can recall that a GIT stable point must have finite stabiliser with respect to the relevant $G$-action.

    \medskip
    Note that any equivalence on $\mathfrak{X}$ lifted from an isomorphism $\Tilde{\tau_{I,I'}}$ does not fix any object unless $\Tilde{\tau_{I,I'}}$ is the identity map. This is clear from the fact that $\Tilde{\tau_{I,I'}}$ acts on a tuple in $\AA^{n+1}$ by changing the position of its zero entries while preserving the relative order of its nonzero entries. The only way to fix a tuple is to leave its zero entries in their original position, but any map $\Tilde{\tau_{I,I'}}$ which does this is just the identity map.
\end{proof}

\begin{corollary}
    The stacks $\mathfrak{M}^m_{\LW}$ and $\mathfrak{M}^m_{\SWS}$ are Deligne-Mumford and proper.
\end{corollary}

\begin{proof}
    This follows directly from the results of this section.
\end{proof}

\subsection{An isomorphism of stacks}\label{isom of stacks}

We shall now show that the stacks $\mathfrak{M}^m_{\SWS}$ and $\mathfrak{M}^m_{\LW}$ are isomorphic. The following lemma is a standard result, quoted from \cite{GHH}.

\begin{lemma}\label{equiv_stacks}
Let $\mathfrak{W}$ and $\mathfrak{Y}$ be Deligne-Mumford stacks of finite type over an algebraically closed field $k$, and let
\[
f \colon \mathfrak{W} \to \mathfrak{Y}
\]
be a representable étale morphism of finite type. Let $|\mathfrak{W}(k)|$ denote the set of equivalence classes of objects in $\mathfrak{W}(k)$ and similarly for $|\mathfrak{Y}(k)|$.

\medskip
Assume that $|f| \colon |\mathfrak{W}(k)| \to |\mathfrak{Y}(k)|$ is bijective and that for every $x\in \mathfrak{W}(k)$, $f$ induces an isomorphism
 \(\Aut_{\mathfrak{W}}(x) \to \Aut_{\mathfrak{Y}}(f(x)). \)
Then $f$ is an isomorphism of stacks.
\end{lemma}

We may construct such a map $f\colon \mathfrak{M}^m_{\SWS} \to \mathfrak{M}^m_{\LW}$, which we will show to have the required properties, in the following way. First recall from above that we have an inclusion $H^m_{[n],\SWS} \subset H^m_{[n],\LW}$. Therefore the natural morphism $H^m_{[n],\LW} \to \mathfrak{M}^m_{\LW}$ restricts to give a morphism $H^m_{[n],\SWS}\to \mathfrak{M}^m_{\LW}$. This morphism is equivariant under the group action so must factor through the morphism
\[
f\colon \mathfrak{M}^m_{\SWS} \longrightarrow \mathfrak{M}^m_{\LW}.
\]

\begin{lemma}\label{bijection of stacks}
The function $|f|\colon |\mathfrak{M}^m_{\SWS}(k)| \to |\mathfrak{M}^m_{\LW}(k)|$ induced by $f$ is a bijection.
\end{lemma}

\begin{proof}
    As we have an inclusion of the SWS stable locus into the LW stable locus, we know that this map must be injective. It remains to show that it is surjective. Let us take any point in $|\mathfrak{M}^m_{\LW}(k)|$. This is given by the equivalence class of a pair $(Z_k,\cX_k)$, where $Z_k$ is a length $m$ zero-dimensional subscheme in a fibre $\cX_k$ over the point $\Spec k$. If the pair $(Z_k,\cX_k)$ is already SWS stable, then there is nothing left to prove. Otherwise, $(Z_k,\cX_k)$ is LW stable but not SWS stable. This implies that there is at least a point of the support in each expanded $\Delta$-component, but there is at least one $\Delta$-component which is not expanded out which contains no point of the support. Let us say this $\Delta$-component is equal to $Y_i$. But by the equivalences of the stack $\mathfrak{X}$ such a fibre is equivalent to a fibre where $Y_i$ is not equal to any $\Delta$-component. It will therefore be equivalent to a fibre in which every $\Delta$-component contains at least one point of the support, which gives us an SWS stable fibre-subscheme pair.
\end{proof}

We will need also the following result from Alper and Kresch \cite{AK}.

\begin{lemma}\label{representable}
Let $\mathfrak{W}$ be a Deligne-Mumford stack with finite inertia, let $\mathfrak{Y}$ be an algebraic stack with separated diagonal and let $f\colon \mathfrak{W}\to \mathfrak{Y}$ be a morphism. Then the largest open substack $\mathfrak{U}$ of $\mathfrak{W}$ on which the restriction of $f$ is a representable morphism enjoys the following characterisation: the geometric points of $\mathfrak{U}$ are precisely those at which $f$ induces an injective homomorphism of stabiliser group schemes.
\end{lemma}

Now we are in a position to prove the following theorem:

\begin{theorem}\label{main_theorem}
The map $f\colon \mathfrak{M}^m_{\SWS} \to \mathfrak{M}^m_{\LW}$ is an isomorphism of stacks.
\end{theorem}

\begin{proof}

This can be seen by applying Lemma \ref{equiv_stacks} to the map $f$. In order to do this we must show that this morphism is representable, with the help of Lemma \ref{representable}. It follows directly from the fact that $\mathfrak{M}^m_{\SWS}$ is a separated Deligne-Mumford stack that it has finite inertia. By Lemma \ref{bijection of stacks}, the first condition of Lemma \ref{equiv_stacks} is satisfied. And the map $f$ defined above must also induce a bijective homomorphism of stabilisers since the only elements which can stabilise a family $(Z,\cX)$ in $\mathfrak{M}^m_{\SWS}$ or $\mathfrak{M}^m_{\LW}$ are elements of $\GG_m^n$ (the other equivalences on $\mathfrak{X}$ do not stabilise any families as explained in Proposition \ref{finite autom}) and, by construction, if a family $(Z,\cX)$ in $\mathfrak{M}^m_{\SWS}$ has stabiliser $\Stab_{(Z,\cX)} \subset \GG_m^n$ then $f((Z,\cX))$ must have the same stabiliser in $\GG_m^n$.
Lemma \ref{equiv_stacks} therefore holds and $f$ is an isomorphism of stacks.

\end{proof}

\printbibliography[
heading=bibintoc,
title={References}
] 

\end{document}